\documentclass[reqno]{amsart}
\usepackage{amssymb, amsmath, amsthm, amscd, mathrsfs}

\usepackage{paralist}
\usepackage{color}
\usepackage{cite}
\usepackage{bigints}
\usepackage{dsfont}
\usepackage{empheq}
\usepackage{amssymb}
\usepackage{cases}
\usepackage{enumitem}
\allowdisplaybreaks

\makeatletter
\renewenvironment{proof}[1][\proofname]{\par
  \pushQED{\qed}%
  \normalfont \topsep6\p@\@plus6\p@\relax
  \trivlist
  \item[\hskip\labelsep\bfseries #1\@addpunct{.}]%
}{%
  \popQED\endtrivlist\@endpefalse
}
\makeatother

\usepackage{caption}
\usepackage{booktabs}

\usepackage{float}
\usepackage[ruled,vlined,linesnumbered,noresetcount]{algorithm2e}

\usepackage[top=3.4cm, bottom=3.4cm, left=4.4cm, right=4.4cm]{geometry}

\usepackage{hyperref}

\newtheorem{theorem}{Theorem}[section]
\newtheorem{lem}[theorem]{Lemma}
\newtheorem{cor}[theorem]{Corollary}
\newtheorem{prop}[theorem]{Proposition}
\newtheorem{defn}[theorem]{Definition}
\newtheorem{rmk}{Remark}
\newtheorem{conj}{Conjecture}

\numberwithin{equation}{section}

\def \d {\mathrm{d}}

\title[Abstract time-fractional Schrödinger equations]{Forward and backward problems for abstract time-fractional Schrödinger equations}

\author{S. E. Chorfi}
\author{F. Et-tahri}
\author{L. Maniar}
\author{M. Yamamoto}

\address{S. E. Chorfi, L. Maniar, Cadi Ayyad University, UCA, Faculty of Sciences Semlalia, Laboratory of Mathematics, Modeling and Automatic Systems, B.P. 2390, Marrakesh, Morocco}
\address{L. Maniar, The Vanguard Center, University Mohammed VI Polytechnic, Benguerir, Morocco}

\address{F. Et-tahri, Faculty of Sciences-Agadir, Lab-SIV, Ibn Zohr University, B.P. 8106, Agadir, Morocco}
\address{M. Yamamoto, Department of Mathematics, Graduate School of Mathematical Sciences, The University of Tokyo, Tokyo, 153-8914, Komaba, Meguro, Japan}
\address{
Department of Mathematics, Faculty of Science, Zonguldak Bulent Ecevit University, Zonguldak, 67100, Turkey
}

\email{s.chorfi@uca.ac.ma, maniar@uca.ac.ma}
\email{fouad.et-tahri@edu.uiz.ac.ma}
\email{myama@ms.u-tokyo.ac.jp}

\makeatletter 
\@namedef{subjclassname@2020}{%
  \textup{2020} Mathematics Subject Classification}
\makeatother

\subjclass[2020]{35R11, 35R30, 35R25}
 \keywords{Fractional Schrödinger equation, backward problem, well-posedness, stability estimate, H\"older stability}

\begin{document}
\begin{abstract}
We investigate forward and backward problems associated with abstract time-fractional Schrödinger equations $\mathrm{i}^\nu \partial_t^\alpha u(t) + A u(t)=0$, $\alpha \in (0,1)\cup (1,2)$ and $\nu\in\{1,\alpha\}$, where $A$ is a self-adjoint operator with compact resolvent on a Hilbert space $H$. This kind of equation, which incorporates the Caputo time-fractional derivative of order $\alpha$, models quantum systems with memory effects and anomalous wave propagation. We first establish the well-posedness of the forward problems in two scenarios: ($\nu=1$,\, $\alpha \in (0,1)$) and ($\nu=\alpha$,\, $\alpha \in (0,1)\cup (1,2)$). Then, we prove well-posedness and stability results for the backward problems depending on the two cases $\nu=1$ and $\nu=\alpha$. Our approach employs the solution's eigenvector expansion along with the properties of the Mittag-Leffler functions, including the distribution of zeros and asymptotic expansions. Finally, we conclude with a discussion of some open problems.
\end{abstract}

\maketitle

\section{Introduction and motivation}
The study of the Schrödinger equation has been central to quantum mechanics and mathematical physics, with applications ranging from wave propagation to quantum computing. In recent years, the fractional generalizations of the Schrödinger equation, which incorporate time-fractional derivatives, have garnered considerable attention due to their ability to model anomalous propagation and memory effects in quantum systems. Such fractional models have found various applications in quantum mechanics, optics and photonics, condensed matter and plasma physics, etc. See, for instance, \cite{Liu23, Na04, Zu21}, the book \cite{La18} and the references therein. The time-fractional Schrödinger equation, typically formulated with the Caputo derivative, extends the classical formulation by introducing a fractional-order parameter $\alpha\in (0,1)\cup (1,2)$, which interpolates between classical and fractional quantum dynamics. Incorporating fractional dynamics disrupts unitarity and time-reversal symmetry, allowing the probability density to either decay or accumulate over time.

A key aspect of studying the Schrödinger equation is understanding its qualitative properties, particularly energy decay, dispersion, and uniqueness. Backward uniqueness, a powerful tool in PDEs, has been widely used to establish unique continuation properties and stability estimates for inverse and ill-posed problems. Fractional derivatives introduce significant mathematical challenges, requiring refined techniques to prove similar properties.

The backward uniqueness property in the context of time-fractional evolution equations is a challenging topic due to the lack of the semigroup law. An initial result was presented in \cite{Ch24} for diffusion-type evolution equations involving self-adjoint operators. The authors even proved a stronger result on logarithmic convexity for the fractional case, which is a relatively recent area of study. This was slightly extended in \cite{Ch25} to include a class of non-symmetric operators, providing some numerical experiments illustrating logarithmic convexity.

To explain the essence of our main contribution, we start with the classical Schrödinger equation with integer derivative:
\begin{equation}\label{iseq}
\begin{cases}
\mathrm{i}\partial_t u +\Delta u =0, &\qquad (t, x) \in (0,T)\times \Omega,\\
u|_{\partial \Omega}=0, &\qquad t \in (0,T),\\
u(0,x)=f(x), &\qquad x\in \Omega,
\end{cases}
\end{equation}
where $\mathrm{i}$ is the imaginary unit and $u : (0,T)\times \Omega \rightarrow\mathbb{C}$ is a complex-valued function describing the wavefunction amplitude associated with the motion of a quantum particle in the bulk $\Omega$. The modulus function $|u(t,\cdot)|^2$ represents the spatial probability density of the particle, up to a constant of proportionality.

The one-parameter family of solution operators $\{S(t) : t \in \mathbb{R}\}$ to \eqref{iseq} forms a $C_0$-unitary group on $(L^2(\Omega),\|\cdot\|)$; in particular
\begin{align*}
\|u(t)\|&=\|f\| \quad\qquad \forall t\in \mathbb{R},\;\forall f\in L^2(\Omega).
\end{align*}
Therefore, the backward well-posedness of \eqref{iseq} is trivial in the integer derivative case. However, in the fractional case, the unitary group property is no longer available, making the backward problems challenging. More precisely, we are concerned with an abstract time-fractional Schrödinger equation posed in a Hilbert space $H$:
\begin{equation}\label{fsweq0}
\begin{cases}
\mathrm{i}^{\nu}\partial_t^\alpha u +Au=0, &\qquad t \in (0,T),\\
\partial_t^k u(0)=u_k, \quad k\in \{0, \lfloor\alpha\rfloor\},
\end{cases}
\end{equation}
where $\nu\in \{1,\alpha\}$, and the Caputo derivative $\partial_t^\alpha u$ of fractional order $\alpha \in (0,1)\cup (1,2)$ is defined by
$$\partial_t^{\alpha} u(t) = \frac{1}{\Gamma(\lfloor\alpha\rfloor+1-\alpha)} \int^t_0 (t-s)^{\lfloor\alpha\rfloor-\alpha} \frac{\d^{\lfloor\alpha\rfloor+1} }{\d s^{\lfloor\alpha\rfloor+1}}u(s) \, \d s,$$
where $\Gamma$ is the standard Gamma function and $\lfloor\alpha\rfloor$ is the floor of $\alpha$.

Using analogy with time-fractional diffusion equations, Naber first introduced the time-fractional Schrödinger equation \cite{Na04} using a core rotation that replaced $\mathrm{i}$ with $\mathrm{i}^\alpha$, leading to the problem \eqref{fsweq0} with $\nu = \alpha$. This modification yields an intuitive physical interpretation for both cases $0<\alpha<1$ and $1<\alpha<2$. In contrast, for $0 < \alpha < 1$, Achar et al. \cite{Ach13} derived the case of \eqref{fsweq0} with $\nu = 1$ using the Feynman path integral approach. This model allows for recovering the classical Schrödinger equation as a limiting case $\alpha\to 1$. On the other hand, the latter model becomes non-physical for $1 < \alpha < 2$, introducing some exponential growth. Based on these facts, we shall consider two cases:
\begin{itemize}
    \item \textbf{Case 1:} $\nu=1$ and $0<\alpha<1$.
    \item \textbf{Case 2:} $\nu=\alpha$, $0<\alpha<1$ or $1<\alpha<2$.
\end{itemize}

\noindent\textbf{Backward problem:} Given $v_k$, $k\in \{0, \lfloor\alpha\rfloor\}$, determine $u_k$ such that the solution to \eqref{fsweq0} satisfies
$$\partial_t^k u(T)=v_k, \qquad k\in \{0, \lfloor\alpha\rfloor\}.$$
We will study the well-posedness of forward and backward problems in two cases:
\begin{itemize}
    \item the subdiffusive case: $0<\alpha<1$,
    \item the superdiffusive case: $1<\alpha<2$.
\end{itemize}

Although there is a rich literature on time-fractional heat and wave equations, see, e.g., \cite{Ce25,Ch25,Ch24,FY20,Hao19}, backward problems have not been studied before for the fractional Schrödinger equations. As for inverse and ill-posed problems for Schrödinger equations with an integer derivative, the literature is abundant, and one can refer, for instance, to the papers \cite{AM12, BP02, CHM25, IY24} and the cited bibliography.

\subsection*{Literature on fractional Schrödinger equations}
There are only a few works regarding the time-fractional Schrödinger equation. We mention the paper \cite{CC23}, where the authors present a unified mathematical framework for modeling and analyzing superdiffusive quantum wave behavior using time-fractional Schrödinger equations that interpolate between the classical Schrödinger equation and the wave equation, capturing non-Markovian dynamics and memory effects. The paper \cite{CC24} addresses the control of quantum systems governed by fractional-in-time Schrödinger equations with the Caputo derivative of order $\alpha\in (1,2)$. On the other hand, the paper \cite{ramswroop14} developed a semi-analytical method combining Laplace and homotopy transforms to solve time-fractional Schrödinger equations with Caputo derivatives, capturing quantum memory effects through rapidly converging series solutions. The article \cite{Gr19} proved the existence of solutions to a nonlinear space-time-fractional equation for the subdiffusive case using Fourier multiplier estimates. Moreover, the paper \cite{Wang07} solves a generalized Schrödinger equation incorporating space-time fractional derivatives, using integral transforms to derive solutions for free particles and potential wells. The author of \cite{Lee20} established Strichartz estimates in the fractional setting. In addition, the paper \cite{Gab23} explores analytical and numerical solutions of a fractional Schrödinger equation with time-dependent potentials. The research \cite{Ba24} investigates a coupled nonlinear system of fractional Schrödinger equations in $\mathbb{R}^n$.

On the other hand, the paper \cite{Su19} proved $L^p-L^q$ estimates for a semilinear space-time fractional equation, while \cite{Su21} addressed some dispersive estimates. As recent related works, we cite \cite{Zh25} and \cite{Zh252} on the local/global well-posedness of space-time fractional Schrödinger equations in the Euclidean space, providing new dispersive and Sobolev estimates. Finally, the paper \cite{murad25} addressed a time-fractional nonlinear Schrödinger equation with applications to soliton dynamics in nonlinear optical media.

Regarding inverse problems for time-fractional Schrödinger equations, the only works we are aware of are \cite{AS23} and \cite{AS22} for $\nu=1$, where the authors studied the determination of a time-dependent source term from an additional condition. They proved the existence and uniqueness of the solution to the inverse problem, along with a stability estimate. Surprisingly, to the best of our knowledge, there are no serious papers on backward problems for the time-fractional Schrödinger equation. In this paper, we initiate the study of such problems by investigating the well-posedness and forward and backward problems for the fractional-in-time Schrödinger equations. We derive conditions under which the solution satisfies the backward well-posedness. Our approach utilizes a series expansion in terms of eigenvectors and properties of the Mittag-Leffler function, among other techniques.

The article is organized as follows. In Section \ref{sec2}, we recall some needed properties of Mittag-Leffler functions, proving new results on the distribution of zeros on the imaginary axis. Section \ref{sec3} presents the proofs of the well-posedness of forward and backward problems in the first case $\nu=1$ for the subdiffusive regime. Then, Section \ref{sec4} is devoted to establishing some well-posedness results for forward and backward problems in the case $\nu=\alpha$ for the subdiffusive and superdiffusive regimes. Finally, in Section \ref{sec5}, we provide some concluding comments and directions for future research.

\section{Preliminary results}\label{sec2}
The study of fractional derivatives is closely related to the following Mittag-Leffler function (introduced first in \cite{ML03}):
\begin{align*}
    E_{\alpha,\beta}(z)=\sum_{n=0}^{\infty}\frac{z^n}{\Gamma(\alpha n+\beta)} \qquad \forall z\in\mathbb{C},
\end{align*}
where $\alpha>0$, $\beta\in \mathbb{R}$. For $\alpha=\beta=1$, we recover the classical exponential function $\mathrm{e}^{z}$. We refer to the books \cite{podlubny99} and \cite{Gor20} for more details.

Recall that a function $f:[0, \infty) \rightarrow \mathbb{R}$ is said to be completely monotone if it has derivatives of all orders satisfying
$$
(-1)^n f^{(n)}(t) \geq 0 \quad \text { for all } t>0,\quad n=0,1,2, \ldots
$$
If $0 <\alpha \le 1$, it is known that the function $E_{\alpha,1}(-t)$ is completely monotone for $t\ge 0$; see \cite{Po48}. Moreover, we have the following lemma. 
\begin{lem}[see \cite{Sc96}]\label{lmcm}
If $0 <\alpha \le 1$ and $\beta \ge \alpha$, the function $E_{\alpha, \beta}(-t)$ is completely monotone for $t\ge 0$.
\end{lem}

We shall use the asymptotic expansions as $|z|\to \infty$ and upper bounds of the Mittag-Leffler function.
\begin{lem}[Pages 32-35 in \cite{podlubny99}]
Let $0<\alpha<2$ and $\beta \in \mathbb{R}$ be arbitrary. Let $\mu$ satisfy $\frac{\pi \alpha}{2}<\mu<\min \{\pi, \pi \alpha\}$. Then, for any $p\ge 1$,
\begin{equation}\label{asymptotic expansions}
    E_{\alpha, \beta}(z)=-\sum_{k=1}^p\frac{1}{\Gamma(\beta-\alpha k)}\frac{1}{z^k} + \mathcal{O}\left(\frac{1}{|z|^{1+p}} \right), \; |z|\to \infty, \; \mu \leq |\arg (z)| \leq \pi;
\end{equation}
\begin{equation}\label{asymptotic expansions2}
    E_{\alpha, \beta}(z)=\frac{1}{\alpha} z^{\frac{1-\beta}{\alpha}} e^{z^{\frac{1}{\alpha}}}-\sum_{k=1}^p\frac{1}{\Gamma(\beta-\alpha k)}\frac{1}{z^k} + \mathcal{O}\left(\frac{1}{|z|^{1+p}} \right), \; |z|\to \infty, \; |\arg (z)|\leq \mu.
\end{equation}
Moreover, there exist positive constants $C,C_1,C_2$ depending on $(\alpha, \beta, \mu)$ such that
\begin{equation}\label{boundedness M-L}
\left|E_{\alpha, \beta}(z)\right| \leq \frac{C}{1+|z|}, \qquad \mu \leq |\arg (z)| \leq \pi.
\end{equation}
\begin{equation}\label{boundedness M-L2}
\left|E_{\alpha, \beta}(z)\right| \leq C_1 (1+|z|)^{\frac{1-\beta}{\alpha}} e^{\mathrm{Re}(z^{\frac{1}{\alpha}})}+ \frac{C_2}{1+|z|}, \qquad |\arg (z)| \leq \mu.
\end{equation}
\end{lem}
As a result, the estimate \eqref{boundedness M-L} for $\arg (z)=-\frac{\pi}{2}$ and the estimate \eqref{boundedness M-L2} for $\arg (z)=-\frac{\pi \alpha}{2}$ yield:
\begin{cor}
\label{est12}
Let $t>0$ and $\lambda\ge 0$.
\begin{enumerate}
    \item If $0<\alpha<1$. There exists a constant $C_0>0$, depending only on $\alpha$ and $\mu \in\left(\frac{\pi \alpha}{2}, \pi \alpha \right)$, such that
\begin{align}\label{es0}
& \left|E_{\alpha, 1}\left(-\mathrm{i} \lambda t^\alpha\right)\right| \leq \frac{C_0}{1+\lambda t^\alpha} \leq C_0.
\end{align}
\item If $0<\alpha<2$. There exist constants $C_1, C_2, C_3>0$, depending only on $\alpha$ and $\mu \in\left(\frac{\pi \alpha}{2},\min(\pi,\pi \alpha) \right)$, such that
\begin{align}
& \left|E_{\alpha, 1}\left(\mathrm{i}^{-\alpha} \lambda t^\alpha\right)\right| \leq C_1+\frac{C_2}{1+\lambda t^\alpha} \leq C_3; \label{es1}\\
& \left|E_{\alpha, 2}\left(\mathrm{i}^{-\alpha} \lambda t^\alpha\right)\right| \leq C_1\left(1+\lambda t^\alpha\right)^{\frac{-1}{\alpha}}+\frac{C_2}{1+\lambda t^\alpha} \leq \frac{C_3}{\left(1+\lambda t^\alpha\right)^{\min(1,\frac{1}{\alpha})}}; \label{es2}\\
& \left|E_{\alpha, \alpha}\left(\mathrm{i}^{-\alpha} \lambda t^\alpha\right)\right| \leq C_1\left(1+\lambda t^\alpha\right)^{\frac{1-\alpha}{\alpha}}+\frac{C_2}{1+\lambda t^\alpha} \leq C_3\left(1+\lambda t^\alpha\right)^{\frac{1-\alpha}{\alpha}}; \label{es3}\\
& \left|E_{\alpha, \alpha-1}\left(\mathrm{i}^{-\alpha} \lambda t^\alpha\right)\right| \leq C_1\left(1+\lambda t^\alpha\right)^{\frac{2-\alpha}{\alpha}}+\frac{C_2}{1+\lambda t^\alpha} \leq C_3\left(1+\lambda t^\alpha\right)^{\frac{2-\alpha}{\alpha}}. \label{es4}
\end{align}
\end{enumerate}
\end{cor}

Using termwise differentiation, we obtain the following lemma, which will be needed in the sequel.
\begin{lem}
    Let $0<\alpha<2$. For all $t>0$ and all $\omega\in\mathbb{C}$, we have
    \begin{align}
        &\partial_t^\alpha E_{\alpha,1}(\omega t^{\alpha})=\omega E_{\alpha,1}(\omega t^{\alpha}), \label{derivative0_ML}\\
        &\partial_t E_{\alpha,1}(\omega t^{\alpha})=\omega t^{\alpha-1}E_{\alpha,\alpha}(\omega t^{\alpha}), \label{derivative1_ML}\\
        &\partial_t \left(tE_{\alpha,2}(\omega t^{\alpha})\right)=E_{\alpha,1}(\omega t^{\alpha}), \label{derivative2_ML}\\
        &\partial_t \left(t^{\alpha-1}E_{\alpha,\alpha}(\omega t^{\alpha})\right)=t^{\alpha-2}E_{\alpha,\alpha-1}(\omega t^{\alpha}). \label{derivative3_ML}
    \end{align}
\end{lem}

Sometimes, to overcome the difficulties coming from the imaginary unit, we will use the following key formula (see \cite{BS05}):
\begin{equation}\label{kf}
E_{\alpha,1}(-\mathrm{i} t)=E_{2 \alpha,1}\left(-t^2\right)-\mathrm{i} t E_{2 \alpha, 1+\alpha}\left(-t^2\right), \quad t\in \mathbb{R}.
\end{equation}

The following proposition studies the zeros of the Mittag-Leffler function along the imaginary axis.
\begin{prop}\label{zeros}
The following assertions hold:
\begin{itemize}
    \item[(i)] If $\alpha\in\left(0,\frac{3}{5}\right]$, then $E_{\alpha,1}(\mathrm{i} t)\neq 0$ for all $t\in\mathbb{R}$;
    \item[(ii)] If $\alpha\in \left(\frac{3}{5}, 1\right)$, there exists at most a finite number of $t\in \mathbb{R}$ such that $E_{\alpha,1}(\mathrm{i} t) \ne 0$.
\end{itemize}
\end{prop}
\begin{proof} $(\mbox{i})$ We proceed by considering two cases: $(\mbox{a})$ $\alpha\in \left(0,\frac{1}{2}\right]$ and $(\mbox{b})$ $\left(\frac{1}{2},\frac{3}{5}\right]$.
\begin{itemize}
    \item[(a)] We assume that $E_{\alpha,1}(\mathrm{i} t_0)=0$ for some $t_0\in\mathbb{R}$. Then, by the key formula \eqref{kf}, we have $E_{2 \alpha,1}\left(-t_0^2\right)=0$. Since $0<2\alpha\le 1$, Lemma \ref{lmcm} implies that the function $E_{2 \alpha,1}\left(-t\right)$ is completely monotone for $t\ge 0$. In particular, $E_{2 \alpha,1}\left(-t\right)$ is non-negative and non-increasing. Therefore, we obtain $E_{2 \alpha,1}\left(-t\right)=0$ for all $t\ge t_0^2$. Holomorphy of $E_{2 \alpha,1}\left(-z\right)$ implies $E_{2 \alpha,1}\left(-t\right)=0$ for all $t> 0$, which is a contradiction with $E_{2 \alpha,1}(0)=1$.
    \item[(b)] Using the key formula \eqref{kf}, it suffices to argue on the imaginary part, i.e., the function $E_{2\alpha,1+\alpha}(t)$. We set $\rho=\frac{1}{2\alpha}\in [\frac{1}{2},1]$. By \cite[Theorem 6.1.1]{PS13}, there exists a function $f(\rho)$ decreasing on the segment $[\frac{1}{2},1]$ such that for any $\lambda > f(\rho)$, the function $E_{2\alpha,\lambda}(t)$ has no real zeros. Moreover, by \cite[Theorem 6.1.3]{PS13} the function $f$ satisfies
$$f(\rho)<\frac{4}{3\rho}, \qquad \frac{2}{3}<\rho<1.$$
Now, we choose $\lambda=1+\alpha\ge\frac{4}{3\rho}>f(\rho)$, because $\alpha\in \left(\frac{1}{2},\frac{3}{5}\right]$. Therefore, the function $E_{2\alpha,1+\alpha}(t)$ has no real zeros.
\end{itemize}
$(\mbox{ii})$ Since $E_{\alpha, 1}(z)$ is analytic, it suffices to prove that there exists a constant $R>0$ such that $E_{\alpha, 1}(\mathrm{i} t)\neq 0$ for all $t\in\mathbb{R}$ such that $|t|\ge R$.

Using the asymptotic expansion \eqref{asymptotic expansions}, we have
$$
E_{\alpha, 1}(\mathrm{i} t)=\frac{\mathrm{i}}{\Gamma(1-\alpha) t} + \mathcal{O}\left(\frac{1}{t^2} \right), \quad |t|\to \infty.
$$
Note that $\Gamma(1-\alpha)>0$ in this case. Then, using the reverse triangle inequality, there exist $R_1>0$ and $M>0$ such that for all $|t| >R_1$, 
\begin{align*}
|E_{\alpha, 1}(\mathrm{i} t)|&\geq \frac{1}{\Gamma(1-\alpha) |t|} -\frac{M}{t^2}\\
&\geq \frac{1}{|t|}\left(\frac{1}{\Gamma(1-\alpha)}-\frac{M}{|t|}\right).
\end{align*}
Therefore, $|E_{\alpha, 1}(\mathrm{i} t)|>0$ for all $|t| > R=\max\{R_1,\Gamma(1-\alpha)M\}$.
\end{proof}

\begin{rmk}\label{rmk3/5}
When $\alpha\in\left(\frac{3}{5}, 1\right)$, the question of whether the Mittag-Leffler function has zeros along the entire imaginary axis is subtle. The complete monotonicity argument is no longer valid. The result above provides partial insight. As shown in Figure \ref{Fig}, we see that the Mittag-Leffler function $E_{\alpha,1}(z)$ has no zeros on the entire imaginary axis. Moreover, by the key formula \eqref{kf}, we can look at the real and imaginary parts of $E_{\alpha,1}(\mathrm{i}t)$. From numerical computation, we see that both can have real zeros, but not simultaneously (see e.g., Figure \ref{Fig_2}).
\end{rmk}

\begin{figure}[H]
\centering
\includegraphics[scale=0.6]{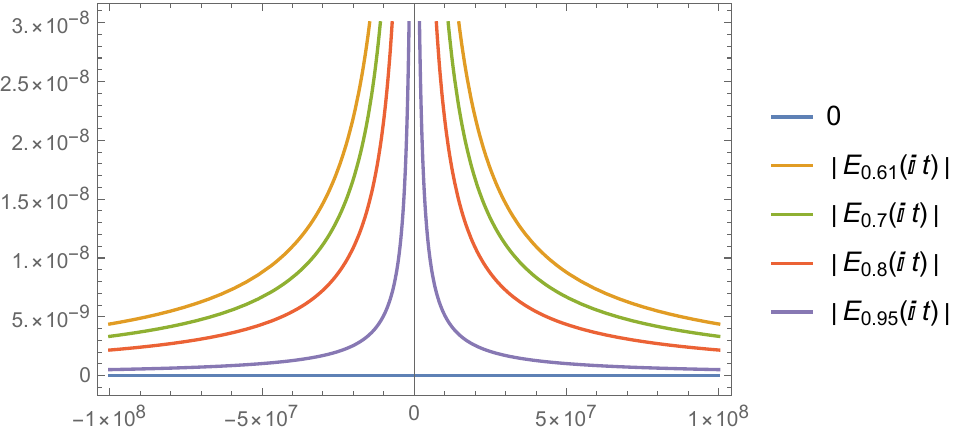}
\caption{$|E_{\alpha,1}(\mathrm{i}t)|$ for $\alpha\in \{0.61,0.7,0.8,0.95\}$.}
\label{Fig}
\end{figure}

\begin{figure}[H]
\centering
\includegraphics[scale=0.6]{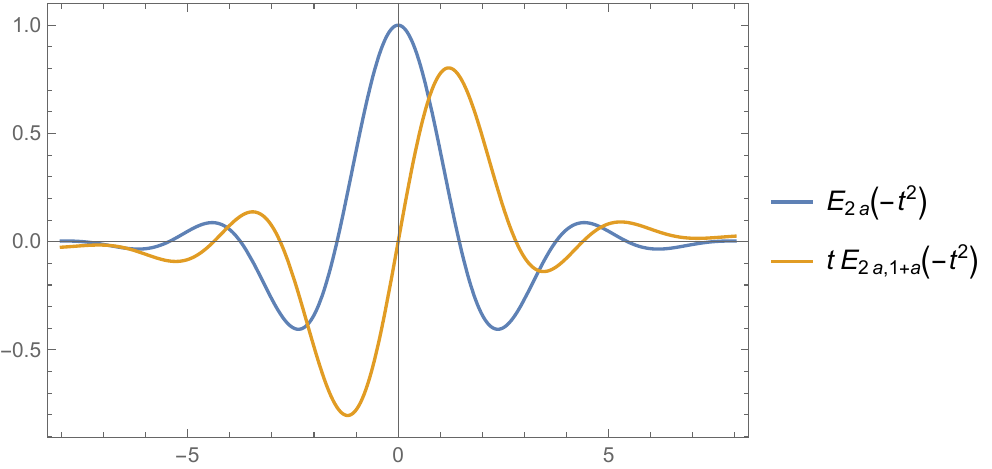}
\caption{$E_{2\alpha,1}(-t^2)$ and $t E_{2\alpha,1+\alpha}(-t^2)$ for $\alpha=0.8$.}
\label{Fig_2}
\end{figure}

We also need the following proposition for the backward problem when $\nu=\alpha$.
\begin{prop}\label{zeros2}
Let $\alpha\in (0,1)\cup (1,2)$. Then $E_{\alpha,1}(\mathrm{i}^{-\alpha} t)\neq 0$ for all $t>0$.
\end{prop}

\begin{proof}
Assume first that $0<\alpha<1$. Then $\rho=\frac{1}{\alpha}>1$. By \cite[Theorem 1-(i)]{OP97}, all zeros of the function $E_{\alpha,1}(z)$ lie outside the angle
\begin{equation}\label{sect}
|\arg z|\leq \frac{\pi \alpha}{2}.
\end{equation}
Since $|\arg (\mathrm{i}^{-\alpha} t)|=\frac{\pi \alpha}{2}$, $\mathrm{i}^{-\alpha} t$ is not a zero of $E_{\alpha,1}(z)$ for all $t>0$.

If $1<\alpha<2$, then $\rho=\frac{1}{\alpha}$ satisfies $\frac{1}{2}<\rho<1$, $1\ge \frac{1}{\rho}-1$, and we have $|\arg (\mathrm{i}^{-\alpha} t)|=\frac{\pi \alpha}{2}$. By \cite[Theorem 1-(iii)]{OP97}, again all zeros of the function $E_{\alpha,1}(z)$ lie outside the angle \eqref{sect}. This completes the proof.
\end{proof}

\section*{Notational setting}
Throughout the paper, we assume that $(H,\langle \cdot,\cdot \rangle)$ is a Hilbert space, $\|\cdot\|:=\sqrt{\langle \cdot,\cdot \rangle}$, $A:D(A)\subset H\rightarrow H$ is a negative-definite self-adjoint operator with compact resolvent.
We set $\sigma(-A) = \{\lambda_n\}_{n\in \mathbb{N}}\subset \mathbb{R}^{+}$ for the spectrum of $-A$ and $\{\varphi_n\}_{n\in \mathbb{N}}\subset D(A)$ for the corresponding Hilbert basis of eigenvectors. In particular, the operator $A$ is such that $$\langle Au,u \rangle <0 \qquad \forall u\in D(A)\setminus \{0\},$$
which means that the spectrum $\sigma(-A) = \{\lambda_n\}_{n\in \mathbb{N}}$ satisfies
$$0<\lambda_1 \le \lambda_2 \le \cdots \le \lambda_n \le \cdots\to \infty.$$
Moreover, for all $s\ge 0$, we denote by $(-A)^s$ the operator defined by
\begin{align*}
\begin{cases}
    D((-A)^s)=H_s,\\
    (-A)^s u=\displaystyle\sum_{n=1}^{\infty}\lambda_{n}^s\langle u, \varphi_n \rangle \varphi_n \qquad \forall u\in H_s,
\end{cases}
\end{align*}
where
\begin{align}
	H_s:=\left\{u\in H\colon \sum_{n=1}^{\infty}\lambda_{n}^{2s}|\langle u, \varphi_n \rangle|^2 <\infty\right\}
\end{align}
is a Hilbert space equipped with the norm $\|\cdot\|_s$:
\begin{align}
	\|u\|_s^2=\sum_{n=1}^{\infty}\lambda_{n}^{2s}|\langle u, \varphi_n \rangle|^2.
\end{align}
Note that $H_0=H$, $\|\cdot\|_0=\|\cdot\|$ and $H_1=D(A)$. By duality and identification $H=H^{\prime}$, we can also set $H_{-s}:=H_s^{\prime}$ which consists
of bounded linear functionals on $H_s$ with the norm $\|\cdot\|_{-s}$:
\begin{align}
	\|u\|_{-s}^2=\sum_{n=1}^{\infty}\lambda_{n}^{-2s}|\langle u, \varphi_n \rangle_{-s,s}|^2.
\end{align}
Here $\langle\cdot,\cdot\rangle_{-s,s}$ denotes the duality bracket between $H_{-s}$ and $H_s$. We further note that:
\begin{align}
	\langle u, v\rangle_{-s,s}=\langle u, v\rangle\quad \mbox{if}\; u\in H\;\mbox{and}\; v\in H_s.
\end{align}

In the sequel, we sometimes use the notation $a \lesssim b$ to indicate that there exists a constant $C > 0$ that is independent of the quantities under consideration, such that $a \le C b$.

\section{The case $\nu=1$ and $0<\alpha<1$}\label{sec3}
\subsection{Well-posedness of the forward problem}
Let $0<\alpha<1$ and consider the following problem
\begin{equation}\label{fseq}
\begin{cases}
\mathrm{i}\partial_t^\alpha u +Au=0, &\qquad t \in (0,T),\\
u(0)=u_0,
\end{cases}
\end{equation}
where $T>0$, and the Caputo derivative $\partial_t^\alpha u$ of fractional order $\alpha\in (0,1)$ is defined by
\begin{equation}\label{cap01}
\partial_t^\alpha u(t)=\frac{1}{\Gamma(1-\alpha)} \int_0^t(t-s)^{-\alpha}\partial_s u(s)\, \d s,
\end{equation}
whenever the right-hand side is well-defined.

Applying the Fourier approach and the Laplace transform in the temporal variable, we obtain the formal solution to \eqref{fseq}:
\begin{align}
	u(t)=\sum_{n=1}^{\infty}\left\langle u_0, \varphi_n\right\rangle E_{\alpha, 1}\left(-\mathrm{i}\lambda_n t^\alpha\right) \varphi_n. \label{solution formula}
	\end{align}
\begin{lem} \label{regularity1}
	Let $s\geq 0$ and $u_0\in H_s$. The solution \eqref{solution formula} satisfies the following regularity properties:
	\begin{itemize}
		\item[(i)] $u\in C([0,T];H_s)$ and $\|u\|_{C([0,T];H_s)}\lesssim \|u_0\|_s$.
		\item[(ii)] $u\in L^2(0,T;H_{s+1/2})$ and $\|u\|_{L^2(0,T;H_{s+1/2})}\lesssim \|u_0\|_s$.
		\item[(iii)] $u\in C((0,T];H_{s+1})$ and $\|u(t)\|_{s+1}\lesssim t^{-\alpha}\|u_0\|_s$.
	\end{itemize}
	\end{lem}
\begin{proof}
	Using formula \eqref{solution formula} and $|E_{\alpha, 1}\left(-\mathrm{i}\lambda_n t^\alpha\right)|\leq C_0$ for all $t\ge 0$ (see \eqref{es0}), we obtain
		\begin{align}
		\|u(t)\|^2_s&=\sum_{n=1}^{\infty}\lambda_{n}^{2s}|\left\langle u_0, \varphi_n\right\rangle|^2 |E_{\alpha, 1}\left(-\mathrm{i}\lambda_n t^\alpha\right)|^2 \nonumber\\
		&\lesssim\sum_{n=1}^{\infty}\lambda_{n}^{2s}|\left\langle u_0, \varphi_n\right\rangle|^2 \nonumber\\
		&=\|u_0\|^{2}_s.\nonumber
	\end{align}
    
	Next, using \eqref{es0} and $\frac{\lambda_{n}t^{\alpha}}{(1+\lambda_{n}t^{\alpha})^2}\leq \frac{1}{4}$ for all $t\ge 0$, we obtain
	\begin{align}
		\|u(t)\|^2_{s+1/2}&=\sum_{n=1}^{\infty}\lambda_{n}^{2s+1}|\left\langle u_0, \varphi_n\right\rangle|^2 |E_{\alpha, 1}\left(-\mathrm{i}\lambda_n t^\alpha\right)|^2 \nonumber\\
		&\lesssim \sum_{n=1}^{\infty}\frac{\lambda_{n}^{2s+1}}{(1+\lambda_{n}t^{\alpha})^2}|\left\langle u_0, \varphi_n\right\rangle|^2  \nonumber\\
		&\lesssim t^{-\alpha}\sum_{n=1}^{\infty}\frac{\lambda_{n}t^{\alpha}}{(1+\lambda_{n}t^{\alpha})^2}\lambda_{n}^{2s}|\left\langle u_0, \varphi_n\right\rangle|^2  \nonumber\\
		&\lesssim t^{-\alpha}\sum_{n=1}^{\infty}\lambda_{n}^{2s}|\left\langle u_0, \varphi_n\right\rangle|^2 \nonumber\\
		&\lesssim t^{-\alpha}\|u_0\|^{2}_s. \nonumber
	\end{align}
	Since $\alpha<1$ then $u\in L^2(0,T;H_{s+1/2})$ and $\|u\|_{L^2(0,T;H_{s+1/2})}\lesssim \|u_0\|_s$.
    
    Let $\delta\in (0,T)$ and $t\in [\delta, T]$. Using \eqref{es0} and $\frac{\lambda_{n}^2t^{2\alpha}}{(1+\lambda_{n}t^{\alpha})^2}\leq 1$, we obtain
	\begin{align}
		\|u(t)\|^2_{s+1}&=\sum_{n=1}^{\infty}\lambda_{n}^{2s+2}|\left\langle u_0, \varphi_n\right\rangle|^2 |E_{\alpha, 1}\left(-\mathrm{i}\lambda_n t^\alpha\right)|^2 \nonumber\\
		&\lesssim \sum_{n=1}^{\infty}\frac{\lambda_{n}^{2s+2}}{(1+\lambda_{n}t^{\alpha})^2}|\left\langle u_0, \varphi_n\right\rangle|^2  \nonumber\\
		&\lesssim t^{-2\alpha}\sum_{n=1}^{\infty}\frac{\lambda_{n}^2t^{2\alpha}}{(1+\lambda_{n}t^{\alpha})^2}\lambda_{n}^{2s}|\left\langle u_0, \varphi_n\right\rangle|^2  \nonumber\\
		&\lesssim t^{-2\alpha}\sum_{n=1}^{\infty}\lambda_{n}^{2s}|\left\langle u_0, \varphi_n\right\rangle|^2 \nonumber\\
		&\lesssim t^{-2\alpha}\|u_0\|^{2}_s  \nonumber\\
		&\lesssim \delta^{-2\alpha}\|u_0\|^{2}_s<\infty. \nonumber
	\end{align}
	Hence $u\in C((0,T];H_{s+1})$ and $\|u(t)\|_{s+1}\lesssim t^{-\alpha}\|u_0\|_s$.
	\end{proof}
Thanks to formula \eqref{solution formula} and the estimate \eqref{es0}, we observe that the fractional derivative of the solution is defined pointwise in $H$. This motivates the following definition.
\begin{defn} \label{solution_def1}
	We say that $u$ is a weak solution to \eqref{fseq} if the following conditions are fulfilled:
	\begin{itemize}
		\item[(i)] $u\in C([0,T];H)$;
		\item[(ii)] $\partial_{t}^{\alpha}u(t),\, Au(t)\in H$ for almost all $t\in (0,T)$, and \eqref{fseq}$_1$ holds in $H$ for almost all $t\in (0,T)$;
		\item[(iii)] The initial datum $u(0)=u_0$ is fulfilled as follows
		$$\lim_{t\to 0^{+}}\|u(t)-u_0\|=0 \qquad \text{for } u_0\in H.$$
	\end{itemize}
\end{defn} 
We now state the existence, uniqueness, and regularity results for system \eqref{fseq}.
\begin{theorem} \label{Thm_well1}
	Let $0<\alpha < 1$ and $s\geq 0$.
	\begin{itemize}
		\item[(i)] If $u_0\in H_s$, there exists a unique weak solution $u \in C([0, T] ; H_s)$ to \eqref{fseq} such that $\partial_t^\alpha u\in C((0, T] ; H_s)$ and the following estimates hold
	\begin{align}
		&\|u\|_{C([0,T];H_s)}\lesssim \|u_0\|_s, \label{estime1}\\ 
		&\|\partial^{\alpha}_tu(t)\|_s\lesssim \frac{1}{t^{\alpha}}\|u_{0}\|_s \qquad \forall t\in (0,T]. \label{estime2}
	\end{align}
    \item[(ii)] If $u_0\in H_{s+1/2}$, the unique weak solution to \eqref{fseq} satisfies $u \in L^2(0,T; H_{s+1})$, $\partial_t^\alpha u \in L^2(0,T; H_s)$ and the following estimate holds
    \begin{align*}
    	\|u\|_{L^2(0,T; H_{s+1})}+ \|\partial_t^\alpha u\|_{L^2(0,T; H_s)}\lesssim \|u_0\|_{s+1/2}. 
    \end{align*}
	\item[(iii)] If $u_0\in H_{s+1}$, the unique weak solution to \eqref{fseq} satisfies $u \in C([0, T] ; H_{s+1})$, $\partial_t^\alpha u \in C([0, T] ; H_s)$ and the following estimate holds
	\begin{align*}
		\|u\|_{C([0,T];H_{s+1})}+ \|\partial_t^\alpha u\|_{C([0,T];H_s)}\lesssim \|u_0\|_{s+1}.
	\end{align*}
\end{itemize}
	Moreover, the unique weak solution to \eqref{fseq} is given by the formula \eqref{solution formula}.
\end{theorem}
\begin{proof}
	We begin by proving the uniqueness of the weak solution to \eqref{fseq}. For that, let $u_0\in H$ and assume that $u$ is a weak solution to \eqref{fseq}. We proceed to decompose 
	$u(t)$ and $u_0$ as follows:
	\begin{align*}
		u(t)&=\sum_{n=1}^{\infty}u_n(t)\varphi_n,\quad t\in [0,T],\\
		u_0&=\sum_{n=1}^{\infty}\left\langle u_0, \varphi_n\right\rangle\varphi_n.
	\end{align*}
    Since $\partial_t^{\alpha}u$ and $Au$ are defined pointwise in $H$.
	Then, formally, we can write:
	\begin{align}
		\partial_t^\alpha u(t)&=\sum_{n=1}^{\infty}\partial_t^\alpha u_n(t)\varphi_n\quad \forall t\in (0,T], \label{serie1}\\
		Au(t)&=\sum_{n=1}^{\infty}-\lambda_{n}u_n(t)\varphi_n\quad \forall t\in (0,T]. \label{serie2}
	\end{align}
	Hence, $u$ is a weak solution to \eqref{fseq} 
if and only if each $u_n$ satisfies the following fractional ODE:
	\begin{equation}\label{ode}
		\begin{cases}
			\partial_t^\alpha u_n(t) + \mathrm{i} \lambda_{n}u_n(t) = 0, &\qquad t \in (0,T),\\
			u_n(0)=\left\langle u_0, \varphi_n\right\rangle.
		\end{cases}
	\end{equation}
	Using the Laplace transform, the unique solution $u_n$ of \eqref{ode} is given by
	\begin{align*}
		u_n(t)=\left\langle u_0, \varphi_n\right\rangle E_{\alpha,1}(-\mathrm{i} \lambda_{n} t^{\alpha}).
	\end{align*}
	Consequently, we obtain the formula \eqref{solution formula}, which guarantees the uniqueness of the solution when it exists. We now show that formula \eqref{solution formula} indeed defines a weak solution to \eqref{fseq}. Using Lemma \ref{regularity1}, we obtain
	\begin{align}
		&\|u(t)\|\lesssim \|u_0\|, \nonumber\\
		&\|\partial_{t}^{\alpha}u(t)\|=\|Au(t)\|=\|u(t)\|_1\lesssim t^{-\alpha}\|u_0\|<\infty \nonumber,
	\end{align} 
	and 
	\begin{align}
		\|u(t)-u_0\|^2=\sum_{n=1}^{\infty}|\left\langle u_0, \varphi_n\right\rangle|^2 |E_{\alpha, 1}\left(-\mathrm{i}\lambda_n t^\alpha\right)-1|^2\lesssim \|u_0\|^2. \nonumber
	\end{align}
	By applying term-by-term differentiation, we obtain
	\begin{align*}
            & u(t)=\sum_{n=1}^{\infty}\left\langle u_0, \varphi_n\right\rangle E_{\alpha, 1}\left(-\mathrm{i}\lambda_n t^\alpha\right) \varphi_n\in H \qquad \forall t\in [0,T],\\
		&\partial_t^\alpha u(t)=-\mathrm{i}\sum_{n=1}^{\infty}\lambda_{n}\langle u_0,\varphi_{n}\rangle E_{\alpha,1}(-\mathrm{i}\lambda_{n}t^{\alpha})\varphi_n\in H\qquad \forall t\in (0,T], \\
		&Au(t)=-\sum_{n=1}^{\infty}\lambda_{n}\langle u_0,\varphi_{n}\rangle E_{\alpha,1}(-\mathrm{i}\lambda_{n}t^{\alpha})\varphi_n\in H\qquad \forall t\in (0,T],\\
		&\lim_{t\to 0^{+}}\|u(t)-u_0\|=0.
	\end{align*}
Consequently, formula \eqref{solution formula} gives the unique weak solution to \eqref{fseq}. To establish the regularity stated in Theorem \ref{Thm_well1}, it suffices to observe that
\begin{align*}
	\|\partial_t^{\alpha}u(t)\|_s=\|Au(t)\|_s=\|u(t)\|_{s+1},\qquad s\ge 0,
\end{align*}
and use Lemma \ref{regularity1}.
\end{proof}

\subsection{Well-posedness of the backward problem}
In this section, we prove the backward uniqueness properties of the solution \eqref{fseq} and stability estimates for the backward problem.
\begin{theorem}[\textbf{Well-posedness}]\label{bu}\
\\{\bf (i) Case $0<\alpha \le \frac{3}{5}$:}
\begin{itemize}
    \item[\bf (a)] {\bf (uniqueness)}
Let $u_0 \in H$. If the solution $u$ to equation \eqref{fseq} 
satisfies $u(T) = 0$, then $u_0 = 0$ and hence $u(t) = 0$ for all 
$t \in [0, T]$.
\item[\bf (b)] {\bf (existence)} 
For any $u_T \in H_1$, there exists a unique weak solution $u \in C([0, T]; H)\cap 
C((0, T] ; H_1)$ to \eqref{fseq} such that $\partial_t^\alpha u, Au 
\in C((0, T] ; H)$ and $u(T)=u_T$;
\item[\bf (c)] {\bf (stability)}
There exist constants $C_1, C_2 >0$ such that for any solution $u$ to 
\eqref{fseq},
\begin{equation}\label{eqthmhstabLip}
C_1\|u(0)\| \le \|u(T)\|_{1} \le C_2 \|u(0)\|.
\end{equation}
\end{itemize}
{\bf (ii) Case $\frac{3}{5} < \alpha < 1$:}
\begin{itemize}
\item[\bf (a)] Assume that $T$ is sufficiently large. 
Then, for any $u_T \in H_1$, there exists a unique weak solution $u \in C([0, T] ; H)
\cap C((0, T]; H_1)$ to \eqref{fseq} such that $\partial_t^\alpha u, Au 
\in C((0, T] ; H)$ and $u(T)=u_T$.
Moreover, there exist constants $C_1, C_2 >0$ such that for any solution 
$u$ to \eqref{fseq}, we have
\begin{equation}\label{eqthmhstabLip2}
C_1\|u(0)\| \le \|u(T)\|_{1} \le C_2 \|u(0)\|.
\end{equation}
\item[\bf (b)] For any $T>0$, we have
$$
\mathrm{dim}\, \{u_0 \in H: \; u(T) =  0\} < \infty.
$$
\end{itemize}
\end{theorem}

The stability estimates \eqref{eqthmhstabLip} and \eqref{eqthmhstabLip2} are unconditional Lipschitz stability estimates with a stronger norm on the final data $u(T)$. We refer to \cite[Theorem 4.1]{SY11} for fractional diffusion equations and 
\cite[Theorem 1.3]{FY20} for fractional wave equations. Moreover for $0 < \alpha \le \frac{3}{5}$, \eqref{eqthmhstabLip} implies the norm equivalence between $\|u(0)\|$ and $\|u(T)\|_{1}$ for any $T>0$. Therefore, we can prove: for $0< T_1 <T_2\leq T$, there exist constants $\widehat{C}_1, \widehat{C}_2 > 0$ such that
$$
\widehat{C}_1 \Vert u(T_1)\Vert_1 \le \Vert u(T_2)\Vert_1 \le \widehat{C}_2 \Vert u(T_1)\Vert_1,
$$
where $\widehat{C}_1$ and $\widehat{C}_2$ are independent of a choice  of solution $u$.

\begin{proof}\
\\{\bf (i) Case $0<\alpha \le \frac{3}{5}$:}\\
{\bf Proof of (a).}\\
By the solution formula \eqref{solution formula},
$$
u(T)=\sum_{n=1}^{\infty}\left\langle u_0, \varphi_n\right\rangle E_{\alpha, 1}\left(-\mathrm{i}\lambda_n T^\alpha\right) \varphi_n.
$$
Then, for all $n \in \mathbb{N}$, we have
\begin{equation}\label{eqOT}
    \left\langle u_0, \varphi_n\right\rangle E_{\alpha, 1}\left(-\mathrm{i}\lambda_n T^\alpha\right)=\left\langle u(T), \varphi_n\right\rangle=0.
\end{equation}
By Proposition \ref{zeros}-(i), we deduce $\left\langle u_0, \varphi_n\right\rangle=0$ for all $n \in \mathbb{N}$, and then $u_0=0$.
\smallskip

\noindent{\bf Proof of (b) and (c).}\\
It suffices to determine an initial datum 
	$u_0\in H$ such that the corresponding solution 
	$u(t)$ to \eqref{fseq} satisfies $u(T)=u_T$.  
	Since $E_{\alpha, 1}\left(-\mathrm{i}\lambda_{n} T^{\alpha}\right)\neq 0$ for all $n\in\mathbb{N}$ (by Proposition \ref{zeros}-(i)), based on equation \eqref{eqOT}, such an initial datum should be given by:
	\begin{align*}
		u_0
		&:=\sum_{n=1}^{\infty}\frac{\langle u_T,\varphi_n\rangle}{E_{\alpha, 1}\left(-\mathrm{i}\lambda_n T^\alpha\right)} \varphi_n.
	\end{align*}
	Let us now justify that $u_0$ is well-defined and belongs to $H$.

Since $|\arg(-\mathrm{i}\lambda_{n} T^{\alpha})|=\frac{\pi}{2}$, using the asymptotic expansion \eqref{asymptotic expansions}, we have
$$
E_{\alpha,1}(-\mathrm{i}\lambda_nT^{\alpha}) = \frac{1}{\Gamma(1-\alpha)\mathrm{i}\lambda_nT^{\alpha}}
+ \mathcal{O}\left( \frac{1}{\lambda_n^2}\right)
$$
for all large $n \in \mathbb{N}$. Since $E_{\alpha, 1}\left(-\mathrm{i}\lambda_{n} T^{\alpha}\right)\neq 0$ (by Proposition \ref{zeros}-(i)), it follows that there exists a constant $C_1>0$, such that
$$
\left|E_{\alpha, 1}\left(-\mathrm{i}\lambda_{n} T^{\alpha}\right)\right|^{-1} \le C_1 (1+\lambda_n T^\alpha)
$$
for all $n \in \mathbb{N}$. Setting $C=C_1\left(\frac{1}{\lambda_{1}}+T^\alpha\right)$, we obtain
\begin{align}\label{minora_Mittag}
\left|E_{\alpha, 1}\left(-\mathrm{i}\lambda_{n} T^{\alpha}\right)\right|^{-1} 
\le C\lambda_n
\end{align}
for all $n \in \mathbb{N}$.
    
    Using \eqref{minora_Mittag}, we obtain
	\begin{align*}
		\|u_0\|^2 &=\sum_{n=1}^{\infty}\frac{|\langle u_T,\varphi_n\rangle|^2}{|E_{\alpha, 1}\left(-\mathrm{i}\lambda_n T^\alpha\right)|^2}\\
		&\le C^2\sum_{n=1}^{\infty} \lambda_{n}^2|\langle u_T,\varphi_n\rangle|^2	\\
		&=C^2 \|Au_T\|<\infty.
	\end{align*}
	Let $u$ be the solution of \eqref{fseq} associated with this choice of $u_0$. By the solution formula \eqref{solution formula}, we have
	\begin{align*}
		u(T)&=\sum_{n=1}^{\infty}\left\langle u_0, \varphi_n\right\rangle E_{\alpha, 1}\left(-\mathrm{i}\lambda_n T^\alpha\right) \varphi_n\\
		&=\sum_{n=1}^{\infty}\left\langle u_T, \varphi_n\right\rangle \varphi_n\\
		&=u_T.
	\end{align*}
	Concerning the inequalities \eqref{eqthmhstabLip}, by estimate \eqref{estime2}, we have that
	\begin{align*}
		\|u(T)\|^2_{1}=\|Au(T)\|^2=\|\partial_t^\alpha u(T)\|^2\leq \frac{(C^{\prime})^2}{T^{2\alpha}}\|u_{0}\|^2,
	\end{align*}
	for some constant $C^{\prime}>0$. The reversed inequality follows from 
	\begin{align*}
		\|u(T)\|^2_{1}=\sum_{n=1}^{\infty}|\left\langle u_0, \varphi_n\right\rangle |^2|E_{\alpha, 1}\left(-\mathrm{i}\lambda_n T^\alpha\right)|^2\lambda_{n}^2
	\end{align*}
	and estimate \eqref{es0}.
	Then 
	\begin{align*}
		\frac{1}{C}\|u(0)\|\le \|u(T)\|_{1}\leq \frac{C^{\prime}}{T^{\alpha}}\|u(0)\|.
	\end{align*}
\bigskip

\noindent{\bf (ii) Case $\frac{3}{5} < \alpha < 1$.}\\
{\bf Proof of (a).}\\
Let $T>0$ be sufficiently large such that
$$
\vert \lambda_n \vert T^{\alpha} > R \qquad \text{for all } n\in\mathbb{N},
$$
where $R>0$ is the same constant in the proof of Proposition \ref{zeros}-(ii). Then, the latter implies that $E_{\alpha,1}(-\mathrm{i} \lambda_n T^{\alpha}) \ne 0$. The rest of the proof is similar to \eqref{eqthmhstabLip}.
\smallskip

\noindent{\bf Proof of (b).}\\
Let $u_0 =\displaystyle \sum_{n=0}^\infty \langle u_0,\varphi_n\rangle \varphi_n\in H$ be such that $u(T) = 0$. Using the formula \eqref{solution formula}, we obtain,
\begin{equation*}
    \left\langle u_0, \varphi_n\right\rangle E_{\alpha, 1}\left(-\mathrm{i}\lambda_n T^\alpha\right)=\left\langle u(T), \varphi_n\right\rangle=0
\end{equation*}
for all $n \in \mathbb{N}$. By Proposition \ref{zeros}-(ii), there is at most a finite number of $n\in \mathbb{N}$ such that $\langle u_0,\varphi_n\rangle\neq 0$. Therefore, $u_0 =\displaystyle \sum_{\text{finite}} \langle u_0,\varphi_n\rangle \varphi_n$. This completes the proof.
\end{proof}

Next, we establish weaker conditional stability of H\"older type, but with a weaker norm. We refer to \cite[Theorem 1.4]{Ch24} for time fractional diffusion-like equations.

Henceforth, fix two constants $\varepsilon>0$ and $M>0$, then consider the set of admissible initial data
\begin{equation}\label{uad}
\mathcal{U}_{\varepsilon,M}:=\{u_0\in H_\varepsilon \colon \|u_0\|_{\varepsilon} \le M\}.
\end{equation}

\begin{theorem}\label{thmhstab}
Let $0<\alpha \le \frac{3}{5}$ and $\beta =\frac{\varepsilon}{\varepsilon + 1}$. Let $u$ be the associated weak solution to \eqref{fseq}. There exists a constant $C=C(\alpha,\lambda_{1})$ such that
\begin{equation}\label{eqthmhstab}
    \|u(0)\| \le C^{\beta} M^{1-\beta} \|u(T)\|^{\beta}
\end{equation}
for all $u_0\in \mathcal{U}_{\varepsilon,M}$.
\end{theorem}

\begin{rmk} Some remarks are in order: 
\begin{itemize}
    \item Theorem \ref{thmhstab} generally remains valid when the operator $A$ additionally possesses a finite number of non-positive eigenvalues. However, we omitted the details for simplicity.
    \item One can prove the optimality of the Hölder rate in such a stability estimate (see, e.g., \cite[Remark 2]{Hao19}).
\end{itemize}
\end{rmk}

\begin{proof}[Proof of Theorem \ref{thmhstab}]
Since $E_{\alpha, 1}\left(-\mathrm{i}\lambda_{n} T^{\alpha}\right)\neq 0$ for all $n\in\mathbb{N}$ (Proposition \ref{zeros}-(i)), using equation \eqref{eqOT} we can write
\begin{align*}
    \|u(0)\|^{2}&= \sum_{n=1}^{\infty}|\langle u_0, \varphi_{n}\rangle|^{2} \\
    & = \sum_{n=1}^{\infty} \frac{|\langle u(T), \varphi_{n}\rangle|^{2}}{\left|E_{\alpha, 1}\left(-\mathrm{i}\lambda_{n} T^{\alpha}\right)\right|^2}.
\end{align*}
Applying Hölder's inequality, we obtain
\begin{align*}
    \|u(0)\|^{2}
    & = \sum_{n=1}^{\infty} \frac{|\langle u(T), \varphi_{n}\rangle|^{2(1-\beta)}}{\left|E_{\alpha, 1}\left(-\mathrm{i}\lambda_{n} T^{\alpha}\right)\right|^{2}} \; |\langle u(T), \varphi_{n}\rangle|^{2\beta}\\
    & \le \left(\sum_{n=1}^{\infty} \frac{|\langle u(T), \varphi_{n}\rangle|^{2}}{\left|E_{\alpha, 1}\left(-\mathrm{i}\lambda_{n} T^{\alpha}\right)\right|^{\frac{2}{1-\beta}}}\right)^{1-\beta} \left(\sum_{n=1}^{\infty} |\langle u(T), \varphi_{n}\rangle|^{2}\right)^\beta\\
    & \le \left(\sum_{n=1}^{\infty} \frac{|\langle u_0, \varphi_{n}\rangle|^{2}}{\left|E_{\alpha, 1}\left(-\mathrm{i}\lambda_{n} T^{\alpha}\right)\right|^{\frac{2\beta}{1-\beta}}}\right)^{1-\beta} \|u(T)\|^{2\beta}.
\end{align*}
Hence, using \eqref{minora_Mittag}, we obtain
\begin{align*}
    \|u(0)\|^{2}& \le C^{2\beta} \left(\sum_{n=1}^{\infty} |\langle u_0, \varphi_{n}\rangle|^{2} \lambda_n^{\frac{2\beta}{1-\beta}}\right)^{1-\beta} \|u(T)\|^{2\beta}\\
    & \le C^{2\beta} \|u_0\|_{\varepsilon}^{2(1-\beta)} \|u(T)\|^{2\beta}.
\end{align*}
Thus, the proof of Theorem \ref{thmhstab} is complete.
\end{proof}

For $\frac{3}{5} < \alpha \le 1$, similarly to Theorem \ref{thmhstab}, we can prove:
\begin{prop}
Assume that $\frac{3}{5}<\alpha <1$, and $T$ is sufficiently large. We set $\beta =\frac{\varepsilon}{\varepsilon + 1}$. Let $u(t)$ be the associated weak solution to \eqref{fseq}. There exists a constant $C=C(\alpha,\lambda_{1})$ such that
\begin{equation}\label{eqthmhstabT}
    \|u(0)\| \le C^{\beta} M^{1-\beta} \|u(T)\|^{\beta}
\end{equation}
for all $u_0\in \mathcal{U}_{\varepsilon,M}$.
\end{prop}

We can provide a more detailed description of $T$ satisfying $E_{\alpha,1}(-\mathrm{i}\lambda_{n} T^{\alpha}) = 0$ for some $n\in \mathbb{N}$ for $\frac{3}{5}<\alpha <1$. By the holomorphy of $E_{\alpha,1}(z)$ and the asymptotics of the zeros
of the Mittag-Leffler function, we know that the set $\{ \eta >0:\; E_{\alpha,1}(-\mathrm{i}\eta) = 0\}$ is a finite set. We remark that it may be empty. Therefore, we can set
$$
\{ \eta>0:\; E_{\alpha,1}(-\mathrm{i}\eta) = 0\} =: \{ \eta_k\}_{1\le k \le N},
$$
with some $N \in \mathbb{N} \cup \{0\}$ and $\eta_1 <\cdots <\eta_N$. We adopt the convention $\{\eta_k\}_{1\le k \le N}=\varnothing$ when $N=0$.
We have
$$
\eta_k = \lambda_n T^{\alpha} \quad \Longleftrightarrow \quad T \in \left\{ \left(\frac{\eta_1}{\lambda_n}\right)^{\frac{1}{\alpha}},\ldots,
\left(\frac{\eta_N}{\lambda_n}\right)^{\frac{1}{\alpha}}\right\}
_{n\in \mathbb{N}}=: \Upsilon.
$$
Then $T\not\in \Upsilon$ if and only if the backward problem is 
well-posed. Following \cite{FY20}, we can prove:
\begin{prop}  We assume that $\Upsilon \neq \varnothing$. Then
\begin{itemize}
\item[(i)] $\Upsilon \subset (0,\infty)$ is a discrete set.
\item[(ii)] $\sup \Upsilon = \left(\frac{\eta_N}{\lambda_1}\right)^{\frac{1}{\alpha}}$.
\item[(iii)] $0$ is a unique accumulation point of $\Upsilon$.
\end{itemize}
\end{prop}

\subsection*{The case $\nu=1$ and $1<\alpha<2$}
We briefly discuss some issues of the case $\nu=1$ for $1<\alpha<2$, as mentioned earlier in the introduction.

Let $1<\alpha<2$ and consider the following problem
\begin{equation}\label{fseqe}
\begin{cases}
\mathrm{i}\partial_t^\alpha u +Au=0, &\qquad t \in (0,T),\\
u(0)=u_0, \qquad\partial_tu(0)=u_1,
\end{cases}
\end{equation}
where the Caputo derivative $\partial_t^\alpha u$ of fractional order $\alpha\in (1,2)$ is defined by
\begin{equation}\label{cap12}
\partial_t^\alpha u(t)=\frac{1}{\Gamma(2-\alpha)} \int_0^t(t-s)^{1-\alpha} \partial_s^2 u(s)\, \d s,
\end{equation}
whenever the right-hand side is well-defined. By the Fourier method and the Laplace transform, we can obtain the formal solution to \eqref{fseqe}:
	\begin{align}\label{solution formulae}
        u(t)=\sum_{n=1}^{\infty}\left\langle u_0, \varphi_n\right\rangle E_{\alpha, 1}\left(-\mathrm{i}\lambda_n t^\alpha\right) \varphi_n + \sum_{n=1}^{\infty}\left\langle u_1, \varphi_n\right\rangle tE_{\alpha, 2}\left(-\mathrm{i}\lambda_n t^\alpha\right) \varphi_n.
	\end{align}
In this case, $|\arg(-\mathrm{i}\lambda_n t^\alpha)|=\frac{\pi}{2}$ and the adequate asymptotic expansion for the first sum is \eqref{asymptotic expansions2}:
$$E_{\alpha, 1}\left(-\mathrm{i}\lambda_n t^\alpha\right) =\frac{1}{\alpha} e^{(-\mathrm{i})^{\frac{1}{\alpha}} \lambda_n^{\frac{1}{\alpha}} t}+\frac{1}{\Gamma(1-\alpha)}\frac{1}{\mathrm{i}\lambda_n t^\alpha} + \mathcal{O}\left(\frac{1}{\lambda_n^{2} t^{2\alpha}}\right), \; n\to \infty.$$
Since $\mathrm{Re}\left((-\mathrm{i})^{\frac{1}{\alpha}} \lambda_n^{\frac{1}{\alpha}} t\right)=\cos(\frac{\pi}{2\alpha})\lambda_n^{\frac{1}{\alpha}} t$ and $1<\alpha<2$, we have $|E_{\alpha, 1}\left(-\mathrm{i}\lambda_n t^\alpha\right)| \to \infty$ as $n\to \infty$. This suggests that the problem \eqref{fseqe} in $H$ is non-physical for $1 < \alpha < 2$.

\section{The case $\nu=\alpha$}\label{sec4}
\subsection{Well-posedness of the forward problem}
\subsubsection{Subdiffusive case}
Let $0<\alpha<1$ and consider the following problem
\begin{equation}\label{fseq2}
\begin{cases}
\mathrm{i}^{\alpha}\partial_t^\alpha u +Au=0, &\qquad t \in (0,T),\\
u(0)=u_0,
\end{cases}
\end{equation}
where the Caputo derivative is defined by \eqref{cap01}. The Fourier approach and the Laplace transform in $t$ lead to
	\begin{align}\label{solution formula2}
		u(t)=\sum_{n=1}^{\infty}\left\langle u_0, \varphi_n\right\rangle E_{\alpha, 1}\left(\mathrm{i}^{-\alpha}\lambda_n t^\alpha\right) \varphi_n,
	\end{align}
which is the formal expression of the solution to \eqref{fseq2}. Using the derivative of the Mittag-Leffler function \eqref{derivative1_ML}, we formally obtain:
\begin{align}\label{derivative}
	\partial_t u(t)=\sum_{n=1}^{\infty}\mathrm{i}^{-\alpha}\lambda_n\left\langle u_0, \varphi_n\right\rangle t^{\alpha-1}E_{\alpha, \alpha}\left(\mathrm{i}^{-\alpha}\lambda_n t^\alpha\right) \varphi_n.
\end{align}
\begin{lem} \label{regularity2}
	Let $s\geq 0$ and $u_0\in H_s$. The solution \eqref{solution formula2} satisfies the following regularity properties:
	\begin{itemize}
		\item[(i)] $u\in C([0,T];H_s)$ and $\|u\|_{C([0,T];H_s)}\lesssim \|u_0\|_s$;
		\item[(ii)] $u\in W^{1,1}\left(0,T; H_{s-\frac{1}{\alpha}}\right)$.
	\end{itemize}
\end{lem}
\begin{proof}
	Using formula \eqref{solution formula2} and $|E_{\alpha, 1}\left(\mathrm{i}^{-\alpha}\lambda_n t^\alpha\right)|\leq C_3$ for all $t\ge 0$ (see \eqref{es1}), we obtain
	\begin{align}
		\|u(t)\|^2_s&=\sum_{n=1}^{\infty}\lambda_{n}^{2s}|\left\langle u_0, \varphi_n\right\rangle|^2 |E_{\alpha, 1}\left(\mathrm{i}^{-\alpha}\lambda_n t^\alpha\right)|^2 \nonumber\\
		&\lesssim\sum_{n=1}^{\infty}\lambda_{n}^{2s}|\left\langle u_0, \varphi_n\right\rangle|^2 \nonumber\\
		&=\|u_0\|^{2}_s.\nonumber
	\end{align}
	Hence $u\in C([0,T];H_s)$ and $\|u\|_{C([0,T];H_s)}\lesssim \|u_0\|_s$. Since $s-\frac{1}{\alpha}<s$, then $u\in C\left([0,T]; H_{s-\frac{1}{\alpha}}\right)$. Using formulas \eqref{derivative} and \eqref{es3}, we have
	\begin{align}
		\|\partial_t u(t)\|^2_{s-\frac{1}{\alpha}}&\leq t^{2(\alpha-1)}\sum_{n=1}^{\infty}\lambda_{n}^{2s+2-\frac{2}{\alpha}}|\langle u_0,\varphi_{n}\rangle|^2|E_{\alpha, \alpha}\left(\mathrm{i}^{-\alpha}\lambda_n t^\alpha\right)|^2  \nonumber\\
		&\lesssim  t^{2(\alpha-1)}\sum_{n=1}^{\infty}\lambda_{n}^{2s+2-\frac{2}{\alpha}}|\langle u_0,\varphi_{n}\rangle|^2\left(1+\lambda_n t^\alpha\right)^{\frac{2(1-\alpha)}{\alpha}}  \nonumber\\
		&\lesssim  \sum_{n=1}^{\infty}\lambda_{n}^{2s}|\langle u_0,\varphi_{n}\rangle|^2\left(\frac{1+\lambda_n t^\alpha}{\lambda_n t^\alpha}\right)^{\frac{2(1-\alpha)}{\alpha}}  \nonumber\\
		&\lesssim  \left(\frac{t^{-\alpha}}{\lambda_1}+1\right)^{\frac{2(1-\alpha)}{\alpha}}\|u_0\|^2_s \nonumber\\
        &\lesssim  \left(t^{2(\alpha-1)}+1\right)\|u_0\|^2_s. \nonumber
	\end{align}
	Then $\partial_t u\in L^{1}\left(0,T; H_{s-\frac{1}{\alpha}}\right)$ due to $\alpha>0$ and therefore $u\in W^{1,1}\left(0,T; H_{s-\frac{1}{\alpha}}\right)$.
\end{proof}
When $u_0 \in H$ is given, the fractional derivative of \eqref{solution formula2} is defined as an element of $L^1\left(0, T; H_{-\frac{1}{\alpha}}\right)$, due to the properties of the convolution product. Consequently, this fractional derivative is well-defined pointwise in $H_{-\frac{1}{\alpha}}$, which motivates the following definition.
\begin{defn} \label{solution_def2}
	We say that $u$ is a weak solution to \eqref{fseq2} if the following conditions are fulfilled:
	\begin{itemize}
		\item[(i)] $u\in C([0,T];H)$;
		\item[(ii)] $\partial_t^\alpha u(t),\, Au(t)\in H_{-\frac{1}{\alpha}}$ for almost all $t\in (0,T)$ and \eqref{fseq2}$_{1}$ holds in $H_{-\frac{1}{\alpha}}$ for almost all $t\in (0,T)$;
		\item[(iii)] The initial datum $u(0)=u_0$ is fulfilled as follows
		$$\lim_{t\to 0^{+}}\|u(t)-u_0\|=0 \quad \text{ for }u_0\in H.$$
	\end{itemize}
\end{defn} 
We now state the existence and uniqueness result for system \eqref{fseq2}.
\begin{theorem} \label{Thm_well2}
	Let $0<\alpha < 1$ and $u_0\in H_s$ for $s\geq 0$. There exists a unique weak solution $u \in C([0, T] ; H_s)$ to \eqref{fseq2} such that 
		\begin{align}
			&\|u\|_{C([0,T];H_s)}\lesssim \|u_0\|_s.
		\end{align}
	Moreover, the unique weak solution to \eqref{fseq2} is given by the formula \eqref{solution formula2}.
\end{theorem}
\begin{proof}
	We begin by proving the uniqueness of the weak solution to \eqref{fseq}. For that, let $u_0\in H$ and assume that $u$ is a weak solution to \eqref{fseq2}. We proceed to decompose 
	$u(t)$ and $u_0$ as follows:
	\begin{align*}
		u(t)&=\sum_{n=1}^{\infty}u_n(t)\varphi_n\quad t\in [0,T],\\
		u_0&=\sum_{n=1}^{\infty}\left\langle u_0, \varphi_n\right\rangle\varphi_n.
	\end{align*}
	Since $\partial_t^{\alpha}u$ and $Au$ are defined pointwise in $H_{-\frac{1}{\alpha}}$.
	Then, by taking the duality pairing $\langle\cdot,\cdot\rangle_{-\frac{1}{\alpha},\frac{1}{\alpha}}$ with $\varphi_n$, we have
	\begin{align*}
		\langle \partial_t^\alpha u(t), \varphi_{n} \rangle_{-\frac{1}{\alpha},\frac{1}{\alpha}}&=\partial_t^\alpha u_n(t),\\
		\langle Au(t), \varphi_{n} \rangle _{-\frac{1}{\alpha},\frac{1}{\alpha}} &=-\lambda_{n}u_n(t)
	\end{align*}
	for almost all $t\in (0,T)$.
	Hence, $u$ is a weak solution to \eqref{fseq2} 
	if and only if each $u_n$ satisfies the following fractional ODE:
	\begin{equation}\label{ode2}
		\begin{cases}
			\partial_t^\alpha u_n(t) - \mathrm{i}^{-\alpha} \lambda_{n}u_n(t) = 0, &\qquad t \in (0,T),\\
			u_n(0)=\left\langle u_0, \varphi_n\right\rangle.
		\end{cases}
	\end{equation}
	Using the Laplace transform, the unique solution $u_n$ of \eqref{ode2} is given by
	\begin{align*}
		u_n(t)=\left\langle u_0, \varphi_n\right\rangle E_{\alpha,1}(\mathrm{i}^{-\alpha} \lambda_{n} t^{\alpha}).
	\end{align*}
	Consequently, we obtain the formula \eqref{solution formula2}, which guarantees the uniqueness of the solution when it exists. We now show that formula \eqref{solution formula2} indeed defines a weak solution to \eqref{fseq2}. We have
	\begin{align}
		&\|u(t)\|\lesssim \|u_0\|, \nonumber\\
		&\|\partial_{t}^{\alpha}u(t)\|_{-\frac{1}{\alpha}}=\|Au(t)\|_{-\frac{1}{\alpha}}=\|u(t)\|_{1-\frac{1}{\alpha}}\lesssim\|u_0\|<\infty \nonumber,
	\end{align} 
	and 
	\begin{align}
		\|u(t)-u_0\|^2=\sum_{n=1}^{\infty}|\left\langle u_0, \varphi_n\right\rangle|^2 |E_{\alpha, 1}\left(\mathrm{i}^{-\alpha}\lambda_n t^\alpha\right)-1|^2\lesssim \|u_0\|^2. \nonumber
	\end{align}
	By termwise differentiation, we obtain
	\begin{align*}
		& u(t)=\sum_{n=1}^{\infty}\left\langle u_0, \varphi_n\right\rangle E_{\alpha, 1}\left(\mathrm{i}^{-\alpha}\lambda_n t^\alpha\right) \varphi_n\in H \qquad \forall t\in [0,T],\\
		&\partial_t^\alpha u(t)=\mathrm{i}^{-\alpha}\sum_{n=1}^{\infty}\lambda_{n}\langle u_0,\varphi_{n}\rangle E_{\alpha,1}(\mathrm{i}^{-\alpha}\lambda_{n}t^{\alpha})\varphi_n\in H_{-\frac{1}{\alpha}}\qquad \forall t\in (0,T], \\
		&Au(t)=-\sum_{n=1}^{\infty}\lambda_{n}\langle u_0,\varphi_{n}\rangle E_{\alpha,1}(\mathrm{i}^{-\alpha}\lambda_{n}t^{\alpha})\varphi_n\in H_{-\frac{1}{\alpha}}\qquad \forall t\in (0,T],\\
		&\lim_{t\to 0^{+}}\|u(t)-u_0\|=0.
	\end{align*}
	Consequently, formula \eqref{solution formula2} gives the unique weak solution to \eqref{fseq2}.
\end{proof}

\subsubsection{Superdiffusive case}
Let $1<\alpha<2$ and consider the following problem
\begin{equation}\label{fseq3}
\begin{cases}
\mathrm{i}^{\alpha}\partial_t^\alpha u +Au=0, &\qquad t \in (0,T),\\
u(0)=u_0, \qquad\partial_tu(0)=u_1,
\end{cases}
\end{equation}
where the Caputo derivative is defined by \eqref{cap12}. The Fourier approach and the Laplace transform in $t$ lead to
	\begin{align}\label{solution formula3}
        u(t)=\sum_{n=1}^{\infty}\left\langle u_0, \varphi_n\right\rangle E_{\alpha, 1}\left(\mathrm{i}^{-\alpha}\lambda_n t^\alpha\right) \varphi_n + \sum_{n=1}^{\infty}\left\langle u_1, \varphi_n\right\rangle tE_{\alpha, 2}\left(\mathrm{i}^{-\alpha}\lambda_n t^\alpha\right) \varphi_n,
	\end{align}
which is the formal expression of the solution to \eqref{fseq3}. In \eqref{solution formula3}, the duality bracket can be taken depending on the regularity of $u_0$ and $u_1$.
Using the derivatives of the Mittag-Leffler function \eqref{derivative1_ML}-\eqref{derivative3_ML}, we formally obtain:
\begin{align}
	\partial_t u(t)=&\sum_{n=1}^{\infty}\mathrm{i}^{-\alpha}\lambda_n t^{\alpha-1}\left\langle u_0, \varphi_n\right\rangle E_{\alpha, \alpha}\left(\mathrm{i}^{-\alpha}\lambda_n t^\alpha\right) \varphi_n+ \sum_{n=1}^{\infty}\left\langle u_1, \varphi_n\right\rangle E_{\alpha, 1}\left(\mathrm{i}^{-\alpha}\lambda_n t^\alpha\right) \varphi_n, \label{D2}\\
	\partial_t^2 u(t)=&\sum_{n=1}^{\infty}\mathrm{i}^{-\alpha}\lambda_n t^{\alpha-2}\left\langle u_0, \varphi_n\right\rangle E_{\alpha, \alpha-1}\left(\mathrm{i}^{-\alpha}\lambda_n t^\alpha\right) \varphi_n \nonumber\\
    &+\sum_{n=1}^{\infty}\mathrm{i}^{-\alpha}\lambda_n t^{\alpha-1}\left\langle u_1, \varphi_n\right\rangle E_{\alpha, \alpha}\left(\mathrm{i}^{-\alpha}\lambda_n t^\alpha\right) \varphi_n. \label{D3}
\end{align}
\begin{lem} \label{regularity3}
	Let $s\geq 0$, $u_0\in H_s$ and $u_1\in H_{s-\frac{1}{\alpha}}$. The solution \eqref{solution formula3} satisfies the following regularity properties:
	\begin{itemize}
		\item[(i)] $u\in C([0,T];H_s)$ and $\|u\|_{C([0,T];H_s)}\lesssim \|u_0\|_s + \|u_1\|_{s-\frac{1}{\alpha}}$;
		\item[(ii)] $u\in W^{2,1}\left(0,T; H_{s-\frac{2}{\alpha}}\right)$.
	\end{itemize}
\end{lem}
\begin{proof}
Using formula \eqref{solution formula3}, $|E_{\alpha, 1}\left(\mathrm{i}^{-\alpha}\lambda_n t^\alpha\right)|\leq C_3$ and $|E_{\alpha, 2}\left(\mathrm{i}^{-\alpha}\lambda_n t^\alpha\right)|\leq \frac{C}{(1+\lambda_{n}t^{\alpha})^\frac{1}{\alpha}}$ (see Corollary \ref{est12}), we obtain
	\begin{align}
		\|u(t)\|^2_s &\lesssim \sum_{n=1}^{\infty}\lambda_{n}^{2s}|\left\langle u_0, \varphi_n\right\rangle|^2 |E_{\alpha, 1}\left(\mathrm{i}^{-\alpha}\lambda_n t^\alpha\right)|^2 + \sum_{n=1}^{\infty}\lambda_{n}^{2s}t^2|\left\langle u_1, \varphi_n\right\rangle|^2 |E_{\alpha, 2}\left(\mathrm{i}^{-\alpha}\lambda_n t^\alpha\right)|^2 \nonumber\\
		&\lesssim \sum_{n=1}^{\infty}\lambda_{n}^{2s}|\left\langle u_0, \varphi_n\right\rangle|^2  + \sum_{n=1}^{\infty}\frac{\lambda_{n}^{2s}t^2}{(1+\lambda_{n}t^{\alpha})^\frac{2}{\alpha}}|\left\langle u_1, \varphi_n\right\rangle|^2  \nonumber\\
		&\lesssim \|u_0\|_s^2  + \sum_{n=1}^{\infty}\lambda_{n}^{2s-\frac{2}{\alpha}}|\left\langle u_1, \varphi_n\right\rangle|^2  \nonumber\\
		&\lesssim \|u_0\|_s^2  + \|u_1\|^2_{s-\frac{1}{\alpha}} \nonumber.
	\end{align}
Hence $u\in C([0,T];H_s)$ and $\|u\|_{C([0,T];H_s)}\lesssim \|u_0\|_s + \|u_1\|_{s-\frac{1}{\alpha}}$. For the regularity of the first derivative, using formula \eqref{D2}, $|E_{\alpha, \alpha}\left(\mathrm{i}^{-\alpha}\lambda_n t^\alpha\right)|\leq C_3\left(1+\lambda_n t^\alpha\right)^{\frac{1-\alpha}{\alpha}}$ and $|E_{\alpha, 1}\left(\mathrm{i}^{-\alpha}\lambda_n t^\alpha\right)|\leq C_3$, we have
	\begin{align}
		\|\partial_t u(t)\|^2_{s-\frac{1}{\alpha}}&\lesssim \sum_{n=1}^{\infty}\lambda_{n}^{2s-\frac{2}{\alpha}+2}t^{2(\alpha-1)}|\langle u_0,\varphi_{n}\rangle|^2|E_{\alpha, \alpha}\left(\mathrm{i}^{-\alpha}\lambda_n t^\alpha\right)|^2 \nonumber\\
        & \quad + \sum_{n=1}^{\infty}\lambda_{n}^{2s-\frac{2}{\alpha}}|\langle u_1,\varphi_{n}\rangle|^2|E_{\alpha, 1}\left(\mathrm{i}^{-\alpha}\lambda_n t^\alpha\right)|^2  \nonumber\\
		&\lesssim \sum_{n=1}^{\infty}\lambda_{n}^{2s-\frac{2}{\alpha}+2}t^{2(\alpha-1)}|\langle u_0,\varphi_{n}\rangle|^2\left(1+\lambda_n t^\alpha\right)^{\frac{2(1-\alpha)}{\alpha}}+ \sum_{n=1}^{\infty}\lambda_{n}^{2s-\frac{2}{\alpha}}|\langle u_1,\varphi_{n}\rangle|^2  \nonumber\\
		&\lesssim \sum_{n=1}^{\infty}\lambda_{n}^{2s}|\langle u_0,\varphi_{n}\rangle|^2\left(\frac{1+\lambda_n t^\alpha}{\lambda_n t^\alpha}\right)^{\frac{2(1-\alpha)}{\alpha}}+\|u_1\|_{s-\frac{1}{\alpha}}^2. \nonumber
	\end{align}
	Since $\frac{1-\alpha}{\alpha}<0$, then 
	\begin{align}
		\|\partial_t u(t)\|^2_{s} &\lesssim\|u_0\|_{s}^2+\|u_1\|_{s-\frac{1}{\alpha}}^2. \nonumber
	\end{align}
	For the regularity of the second derivative, note that the second sum in $\partial_t^2 u(t)$ (see \eqref{D3}) is similar to the first sum in $\partial_t u(t)$. Using the estimate \eqref{es4}, we have
	\begin{align}
		\|\partial_t^2 u(t)\|^2_{s-\frac{2}{\alpha}}&\lesssim \sum_{n=1}^{\infty}\lambda_{n}^{2s-\frac{4}{\alpha}+2}t^{2(\alpha-2)}|\langle u_0,\varphi_{n}\rangle|^2|E_{\alpha, \alpha-1}\left(\mathrm{i}^{-\alpha}\lambda_n t^\alpha\right)|^2+\|u_1\|_{s-\frac{1}{\alpha}}^2  \nonumber\\
		&\lesssim \sum_{n=1}^{\infty}\lambda_{n}^{2s-\frac{4}{\alpha}+2}t^{2(\alpha-2)}|\langle u_0,\varphi_{n}\rangle|^2\left(1+\lambda_n t^\alpha\right)^{\frac{2(2-\alpha)}{\alpha}}+\|u_1\|_{s-\frac{1}{\alpha}}^2  \nonumber\\
		&\lesssim \sum_{n=1}^{\infty}\lambda_{n}^{2s}|\langle u_0,\varphi_{n}\rangle|^2\left(\frac{1+\lambda_n t^\alpha}{\lambda_n t^\alpha}\right)^{\frac{2(2-\alpha)}{\alpha}}+\|u_1\|_{s-\frac{1}{\alpha}}^2.  \nonumber
	\end{align}
	Since $\left(\frac{1+\lambda_n t^\alpha}{\lambda_n t^\alpha}\right)^{\frac{2(2-\alpha)}{\alpha}}\leq \left(\frac{t^{-\alpha}}{\lambda_{1}}+1\right)^{\frac{2(2-\alpha)}{\alpha}}$, then 
	\begin{align}
		\|\partial_t^2 u(t)\|^2_{s-\frac{2}{\alpha}} &\lesssim \left(\frac{t^{-\alpha}}{\lambda_{1}}+1\right)^{\frac{2(2-\alpha)}{\alpha}}\|u_0\|_{s}^2+\|u_1\|_{s-\frac{1}{\alpha}}^2\nonumber\\
		&\lesssim \left(t^{2(\alpha-2)}+1\right)\|u_0\|_{s}^2+\|u_1\|_{s-\frac{1}{\alpha}}^2.  \nonumber
	\end{align}
	Since $\alpha>1$, $\partial_t^2 u\in L^1\left(0,T; H_{s-\frac{2}{\alpha}}\right)$ and therefore $u\in W^{2,1}\left(0,T; H_{s-\frac{2}{\alpha}}\right)$.
\end{proof}
When $u_0 \in H$ is given, $\partial_t^2 u\in L^1\left(0,T; H_{-\frac{2}{\alpha}}\right)$ and then the fractional derivative of \eqref{solution formula3} is defined as an element of $L^1\left(0, T; H_{-\frac{2}{\alpha}}\right)$, due to the properties of the convolution product. Consequently, this fractional derivative is well-defined pointwise in $H_{-\frac{2}{\alpha}}$, which motivates the following definition.
\begin{defn} \label{solution_def3}
	We say that $u$ is a weak solution to \eqref{fseq3} if the following conditions are fulfilled:
	\begin{itemize}
		\item[(i)] $u\in C([0,T];H)$;
		\item[(ii)] $\partial_t^\alpha u(t),\, Au(t)\in H_{-\frac{2}{\alpha}}$ for almost all $t\in (0,T)$ and \eqref{fseq3}$_{1}$ holds in $H_{-\frac{2}{\alpha}}$ for almost all $t\in (0,T)$;
		\item[(iii)] The initial data $u(0)=u_0$ and $\partial_t u(0)=u_1$ are fulfilled as follows
		$$\lim_{t\to 0^{+}}\|u(t)-u_0\|=0\quad\mbox{and}\quad \lim_{t\to 0^{+}}\|\partial_t u(t)-u_1\|_{-\frac{1}{\alpha}}=0$$
		for $u_0\in H$ and $u_1\in H_{-\frac{1}{\alpha}}$.
	\end{itemize}
\end{defn} 
Using the same argument as in the proof of Theorem \ref{Thm_well2}, we derive the following existence and uniqueness result for system \eqref{fseq3}.
\begin{theorem} \label{Thm_well3}
	Let $1<\alpha < 2$. Let $u_0\in H_s$ and $u_1\in H_{s-\frac{1}{\alpha}}$ for $s\geq 0$. There exists a unique weak solution $u \in C([0, T] ; H_s)$ to \eqref{fseq3} such that 
		\begin{align}
			&\|u\|_{C([0,T];H_s)}\lesssim \|u_0\|_s + \|u_1\|_{s-\frac{1}{\alpha}}.
		\end{align}
	Moreover, the unique weak solution to \eqref{fseq3} is given by the formula \eqref{solution formula3}.
\end{theorem}

\subsection{Well-posedness of the backward problem}
\subsubsection{Subdiffusive case}
Similarly to Theorem \ref{bu}, we can prove the following backward uniqueness result.
\begin{theorem}[\textbf{Uniqueness}]\label{bu2}
Let $0<\alpha < 1$ and $u_0 \in H$. If the solution $u$ to equation \eqref{fseq3} satisfies $u(T) = 0$, then $u_0 = 0$ and hence $u(t) = 0$ for all $t \in [0, T]$.
\end{theorem}

Moreover, we have:
\begin{theorem}\label{thmhstabLip22}
Assume that $0<\alpha <1$. Then
\begin{itemize}
    \item[(i)] for any $u_T \in H_1$, there exists a weak solution $u \in C([0, T]; H)$ to \eqref{fseq3} such that $u(T)=u_T$;
    \item[(ii)] there exist constants $C>0$ such that for any solution $u$ to \eqref{fseq3},
		\begin{equation}\label{eqthmlstabLip2}
			\|u(0)\| \le C \|u(T)\|_{1}.
		\end{equation}
\end{itemize}
\end{theorem}
The proof is similar to Theorem \ref{bu}, but here $|\arg(\mathrm{i}^{-\alpha}\lambda_{n} T^{\alpha})|=\frac{\pi \alpha}{2}$ suggests using the asymptotic expansion \eqref{asymptotic expansions2} to obtain
$$
E_{\alpha,1}(\mathrm{i}^{-\alpha}\lambda_nT^{\alpha}) =\frac{1}{\alpha} e^{-\mathrm{i}\lambda_n^{\frac{1}{\alpha}}T}- \frac{1}{\Gamma(1-\alpha)\mathrm{i}^{-\alpha}\lambda_nT^{\alpha}}
+ \mathcal{O}\left(\frac{1}{\lambda_n^2}\right)
$$
for all large $n \in \mathbb{N}$. Therefore, we can deduce
\begin{align*}
\left|E_{\alpha, 1}\left(\mathrm{i}^{-\alpha}\lambda_{n} T^{\alpha}\right)\right|^{-1} 
\le C\lambda_n. 
\end{align*}

We recall that $\mathcal{U}_{\varepsilon,M}$ is defined by \eqref{uad}. Since $E_{\alpha, 1}\left(\mathrm{i}^{-\alpha}\lambda_{n} T^{\alpha}\right)\neq 0$ for all $n\in\mathbb{N}$ by Proposition \ref{zeros}, similarly to Theorem \ref{thmhstab}, we can prove:
\begin{theorem}\label{thmhstab2}
Let $0<\alpha <1$ and $\beta =\frac{\varepsilon}{\varepsilon + 1}$. Let $u$ be the associated weak solution to \eqref{fseq3}. There exists a constant $C=C(\alpha)$ such that
\begin{equation}\label{eqthmhstab2}
    \|u(0)\| \le C^{\beta} M^{1-\beta} \|u(T)\|^{\beta}
\end{equation}
for all $u_0\in \mathcal{U}_{\varepsilon,M}$.
\end{theorem}

\subsubsection{Superdiffusive case}
Let $1<\alpha<2$ and consider the following backward problem: For given $T>0$ and $u_T,v_T\in H$, determine $u(t)$ solution to:
\begin{equation}\label{fsweq}
\begin{cases}
\mathrm{i}^{\alpha}\partial_t^\alpha u +Au=0, &\qquad t \in (0,T),\\
u(0)=u_0, \qquad \partial_t u(0)=u_1.
\end{cases}
\end{equation}
To study the backward problem, we need to establish some properties of the following function
\begin{equation}\label{psidf}
\psi(t)=E_{\alpha, 1}(\mathrm{i}^{-\alpha}t)^2-\mathrm{i}^{-\alpha}t E_{\alpha, 2}(\mathrm{i}^{-\alpha}t) E_{\alpha, \alpha}(\mathrm{i}^{-\alpha}t), \quad t>0.
\end{equation}

\begin{lem}
The function $\psi(t)$ satisfies the following asymptotic expansion
\begin{equation}\label{asypsi}
\psi(t)=\frac{\mathrm{i}^{\alpha-1}}{\alpha\Gamma(2-\alpha)} t^{\frac{1}{\alpha}-1}e^{-\mathrm{i}t^{\frac{1}{\alpha}}}-\frac{2 \mathrm{i}^\alpha}{\alpha\Gamma(1-\alpha)t}e^{-\mathrm{i}t^{\frac{1}{\alpha}}}+\mathcal{O}\left(\frac{1}{t^{1+\frac{1}{\alpha}}}\right), \; t\to \infty.
\end{equation}
\end{lem}

\begin{proof}
Since $\alpha>1$, it suffices to apply the formula \eqref{asymptotic expansions2} to obtain a second-order asymptotic expansion for each term:
\begin{equation}\label{asyE}
\begin{aligned}
E_{\alpha, 1}(\mathrm{i}^{-\alpha}t)&=\frac{1}{\alpha} e^{-\mathrm{i}t^{\frac{1}{\alpha}}}-\frac{1}{\Gamma(1-\alpha) \mathrm{i}^{-\alpha}t}+\mathcal{O}\left(\frac{1}{t^2}\right), \\
E_{\alpha, 2}(\mathrm{i}^{-\alpha}t)&=\frac{1}{\alpha} \mathrm{i}t^{\frac{-1}{\alpha}} e^{-\mathrm{i}t^{\frac{1}{\alpha}}}-\frac{1}{\Gamma(2-\alpha) \mathrm{i}^{-\alpha}t}+\mathcal{O}\left(\frac{1}{t^2}\right), \\
tE_{\alpha, \alpha}(\mathrm{i}^{-\alpha}t)&=\frac{1}{\alpha} \mathrm{i}^{\alpha-1}t^{\frac{1}{\alpha}} e^{-\mathrm{i}t^{\frac{1}{\alpha}}}-\frac{1}{\Gamma(-\alpha) \mathrm{i}^{-2\alpha}t} +\mathcal{O}\left(\frac{1}{t^2}\right), \; t\to \infty.
\end{aligned}
\end{equation}
A simple calculation then yields the result.
\end{proof}

Using the asymptotic expansion \eqref{asypsi} and similarly to the proof of Proposition \ref{zeros}-(ii), we obtain
\begin{lem}\label{zerospsi}
There exists a constant $R>0$ such that $\psi(t)\neq 0$ for all $t> R$.
\end{lem}

By analyticity of $\psi(t)$ and Lemma \ref{zerospsi}, we know that 
the set $\{ t>0:\; \psi(t) = 0\}$ is a finite set (it might be empty).
Therefore, we can set
$$
\{t>0:\; \psi(t) = 0\} =: \{ t_k\}_{1\le k \le N},
$$
with some $N \in \mathbb{N} \cup \{0\}$ and $t_1 <\cdots <t_N$. We set
\begin{equation}
\Lambda:=\bigcup_{n=1}^\infty\left\{ \left(\frac{t_1}{\lambda_n}\right)^{\frac{1}{\alpha}},\ldots,
\left(\frac{t_N}{\lambda_n}\right)^{\frac{1}{\alpha}}\right\}.
\end{equation}
We notice that $\Lambda$ may be empty.

We have $t_k = \lambda_n T^{\alpha}$ for some $1\le k \le N$ if and only if $T \in \Lambda.$

\begin{prop} We assume that $\Lambda \neq \varnothing$. Then
\begin{itemize}
\item[(i)] $\Lambda \subset (0,\infty)$ is a discrete set.
\item[(ii)] $\sup \Lambda = \left(\frac{t_N}{\lambda_1}\right)^{\frac{1}{\alpha}}$.
\item[(iii)] $0$ is a unique accumulation point of $\Lambda$.
\end{itemize}
\end{prop}

\begin{rmk}\label{psize}
For $1<\alpha<2$, we do not know if the function $\psi(\alpha,t):=\psi(t)$ has any zeros $t>0$. From Figure \ref{Fig_1}, we can see that $\psi(t)$ has no zeros on $(0,\infty)$. The real and imaginary parts of $\psi(t)$ both can have zeros on $(0,\infty)$, but not simultaneously (see Figure \ref{Fig_22}).
\end{rmk}

\begin{figure}[H]
\centering
\includegraphics[scale=0.6]{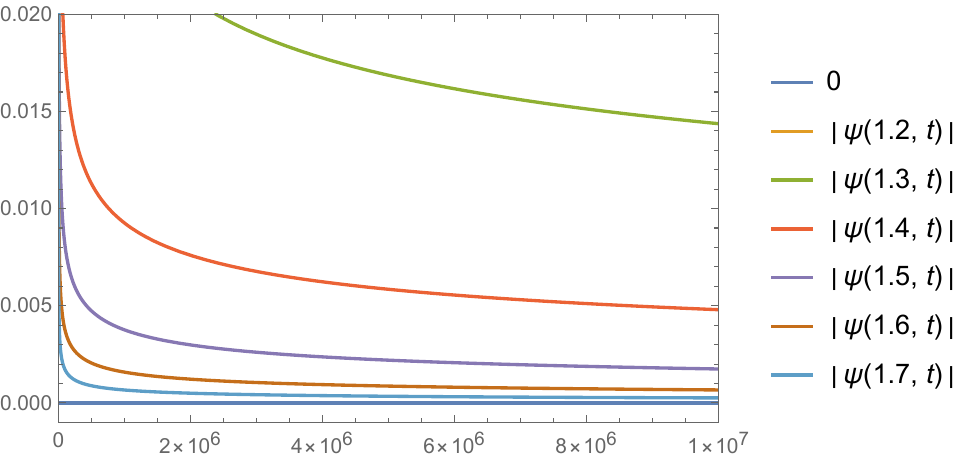}
\caption{$|\psi(t)|$ for $\alpha\in \{1.2,1.3,1.4,1.5,1.6,1.7\}$.}
\label{Fig_1}
\end{figure}

\begin{figure}[H]
\centering
\includegraphics[scale=0.6]{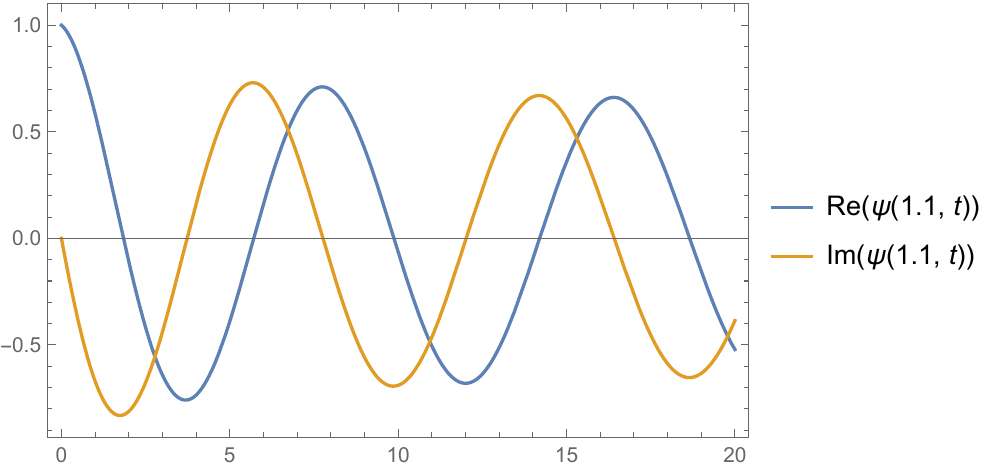}
\caption{$\mathrm{Re}\, \psi(t)$ and $\mathrm{Im}\, \psi(t)$ for $\alpha=1.1$.}
\label{Fig_22}
\end{figure}

Next, we state the main result on the well-posedness of the backward problem.
\begin{theorem}
Assume that $T\notin \Lambda$. For any $u_T, v_T \in H_1$, there exist $u_0, u_1 \in H$ such that the solution $u$ to \eqref{fsweq} satisfies
$$
u(T)=u_T, \qquad \partial_t u(T)=v_T.
$$
Moreover, there exists a constant $C>0$ such that
\begin{equation}
\|u(0)\|+\|\partial_t u(0)\| \leq C\left(\left\|u(T)\right\|_{1}+\left\|\partial_t u(T)\right\|_{1}\right)
\end{equation}
for all $u_T, v_T \in H_1$.
\end{theorem}
\begin{proof}
We seek initial data $u_0, u_1 \in H$ such that the corresponding solution $u$ to \eqref{fsweq} satisfies $u(T)=u_T$ and $\partial_t u(T)=v_T.$ To do this, it is enough to find the coordinates of $u_0$ and $u_1$:
$$
u_{0n}=\langle u_0, \varphi_{n}\rangle, \quad u_{1n}=\langle u_1, \varphi_{n}\rangle.
$$
According to formulas \eqref{solution formula3} and \eqref{D2}, we have:
\begin{equation}\label{eq0}
\begin{aligned}
\|u(T)\|_{1}^2+\left\|\partial_t u(T)\right\|_{1}^2
= \sum_{n=1}^{\infty}  \lambda_n^2\left(\left|a_n\right|^2+\left|b_n\right|^2\right),
\end{aligned}
\end{equation}
where 
\begin{equation}\label{sys}
\left\{\begin{array}{l}
a_{n}:=u_{0n} E_{\alpha, 1}\left(\mathrm{i}^{-\alpha}\lambda_n T^\alpha\right)+u_{1n} T E_{\alpha, 2}\left(\mathrm{i}^{-\alpha}\lambda_n T^\alpha\right), \\
b_{n}:=\mathrm{i}^{-\alpha}\lambda_n T^{\alpha-1} u_{0n} E_{\alpha, \alpha}\left(\mathrm{i}^{-\alpha}\lambda_n T^\alpha\right)+u_{1n} E_{\alpha, 1}\left(\mathrm{i}^{-\alpha}\lambda_n T^\alpha\right).
\end{array}\right.
\end{equation}
Since $\psi\left(\lambda_n T^\alpha\right) \neq 0$ for all $n \in \mathbb{N}$ ($T\notin \Lambda$), we can solve \eqref{sys} with respect to $u_{0n}$ and $u_{1n}$ (determinant is nonzero):
$$
\left\{\begin{array}{l}
u_{0n}=\frac{1}{\psi\left(\lambda_n T^\alpha\right)}\left(a_n E_{\alpha, 1}\left(\mathrm{i}^{-\alpha}\lambda_n T^\alpha\right)-b_n T E_{\alpha, 2}\left(\mathrm{i}^{-\alpha}\lambda_n T^\alpha\right)\right) \\
u_{1n}=\frac{1}{\psi\left(\lambda_n T^\alpha\right)}\left(a_n \mathrm{i}^{-\alpha}\lambda_n T^{\alpha-1} E_{\alpha, \alpha}\left(\mathrm{i}^{-\alpha}\lambda_n T^\alpha\right)+b_n E_{\alpha, 1}\left(\mathrm{i}^{-\alpha}\lambda_n T^\alpha\right)\right).
\end{array}\right.
$$
By \eqref{asyE} and \eqref{asypsi}, we can choose large constants $M>0$ such that
$$
\begin{gathered}
\left|E_{\alpha, 1}(\mathrm{i}^{-\alpha}t)\right| \leq \frac{1}{\alpha}+\frac{M_1}{t}, \quad\left|E_{\alpha, 2}(\mathrm{i}^{-\alpha}t)\right| \leq \frac{1}{\alpha} t^{-\frac{1}{\alpha}} + \frac{M_2}{t}, \\
\left|E_{\alpha, \alpha}(\mathrm{i}^{-\alpha}t)\right| \leq \frac{1}{\alpha} t^{\frac{1}{\alpha}-1} + \frac{M_3}{t^2}, \quad|\psi(t)| \geq \frac{1}{2\alpha \Gamma(2-\alpha)} t^{\frac{1}{\alpha}-1}, \quad t \geq M.
\end{gathered}
$$
Note that $\Gamma(2-\alpha)>0$. By fixing $N_0 \in \mathbb{N}$ and using \eqref{es1}-\eqref{es3},
\begin{align}
&\left|\psi\left(\lambda_n T^\alpha\right)\right| \geq \frac{1}{2 T^{\alpha-1} \Gamma(2-\alpha)} \frac{1}{\lambda_n^{1-\frac{1}{\alpha}}}, \quad n \geq N_0, \quad \left|E_{\alpha, 1}\left(\mathrm{i}^{-\alpha}\lambda_n T^\alpha\right)\right| \le C,\nonumber\\
&\left|E_{\alpha, 2}\left(\mathrm{i}^{-\alpha}\lambda_n T^\alpha\right)\right| \le \frac{C}{\lambda_n^{\frac{1}{\alpha}}}, \quad\left|\lambda_n E_{\alpha, \alpha}\left(\mathrm{i}^{-\alpha}\lambda_n T^\alpha\right)\right| \leq C \lambda_n^{\frac{1}{\alpha}}, \quad n\in \mathbb{N}. \label{estEa}
\end{align}
Therefore,
\begin{equation}\label{esu0u1}
\left|u_{0n}\right| \leq C \lambda_n^{1-\frac{1}{\alpha}}\left( \left|a_n\right|+\left|b_n\right|\right), \quad\left|u_{1n}\right| \leq C \lambda_n\left(\left|a_n\right|+\left|b_n\right|\right), \quad n \geq N_0.
\end{equation}
Since $\psi\left(\lambda_n T^\alpha\right) \neq 0$ for each $n \in \mathbb{N}$, we have
$$
\begin{aligned}
& \left|u_{0n}\right| \leq C \left(\left|a_n\right|+\left|b_n\right|\right)\max_{1 \leq n \leq N_0-1}\left|\frac{1}{\psi\left(\lambda_n T^\alpha\right)}\right|, \\
& \left|u_{1n}\right| \leq C \left(\lambda_n^{\frac{1}{\alpha}}\left|a_n\right|+\left|b_n\right|\right)\max_{1 \leq n \leq N_0-1}\left|\frac{1}{\psi\left(\lambda_n T^\alpha\right)}\right|, \qquad 1 \leq n \leq N_0-1,
\end{aligned}
$$
Then \eqref{esu0u1} is satisfied for all $n \in \mathbb{N}$. Hence
$$
\sum_{n=1}^{\infty} \left(\left|u_{0n}\right|^2+\left|u_{1n}\right|^2\right) \leq C \sum_{n=1}^{\infty}  \lambda_n^2\left(\left|a_n\right|^2+\left|b_n\right|^2\right).
$$
Using \eqref{eq0}, we obtain
$$
\|u_0\|^2+\|u_1\|^2 \leq C \left(\|u(T)\|_{1}^2+\left\|\partial_t u(T)\right\|_{1}^2\right) .
$$
\end{proof}

\section{Conclusion and open problems}\label{sec5}
We have investigated the well-posedness of forward problems associated with the fractional-in-time Schrödinger equations of order $\alpha \in (0,1)\cup (1,2)$. We have considered two models: $\mathrm{i} \partial_t^\alpha u(t) + A u(t)=0$ for $\nu=1$ and $\mathrm{i}^\alpha \partial_t^\alpha u(t) + A u(t)=0$ for $\nu=\alpha$, where each model has different properties that are relevant for physical applications. Subsequently, we address the well-posedness of the corresponding backward problems for both cases. We proved uniqueness and stability estimates for these problems under suitable assumptions. Although the results for the standard Schrödinger equation with integer derivatives are trivial, the fractional setting is rather challenging due to non-unitarity, non-Markovianity, and irreversibility.

Several promising directions for future investigation emerge from this study. We begin by discussing some open problems related to the zeros of the Mittag-Leffler function.

In Proposition \ref{zeros}, we have seen that the Mittag-Leffler function $E_{\alpha,1}(z)$ has no zeros on the imaginary axis for all $0<\alpha\le \frac{3}{5}$. We expect that this would be true for all $0<\alpha<1$. Based on Remark \ref{rmk3/5} and several numerical experiments, we have the following conjecture:
\begin{conj}
For $\alpha\in\left(\frac{3}{5}, 1\right)$, we conjecture that the Mittag-Leffler function $E_{\alpha,1}(z)$ has no zeros on the imaginary axis, i.e., $E_{\alpha,1}(\mathrm{i} t)\neq 0$ for all $t\in\mathbb{R}$.
\end{conj}
If this conjecture is true, the backward problem for the fractional Schrödinger equation when $\nu=1$ and $0<\alpha<1$ will be well-posed for all $T>0$ (not necessarily large as in Theorem \ref{bu}, Case (ii)). Note that most works in the literature concern the real zeros of the Mittag-Leffler functions. We refer to \cite{Se0,Se04,Ps05,PS13} for more details on this topic.

Similarly for the function $\psi(t)$ defined in \eqref{psidf} and based on Remark \ref{psize}, we have the following conjecture:
\begin{conj}
For $\alpha\in (1,2)$, we conjecture that the function $\psi(t)$ has no zeros on $(0,\infty)$, i.e., $\psi(t)\neq 0$ for all $t>0$.
\end{conj}
If this conjecture is true, the backward problem for the fractional Schrödinger equation when $\nu=\alpha$ and $1<\alpha<2$ will be well-posed for all $T>0$. This would be different from the fractional diffusion-wave equation, for which we know the existence of zeros of the corresponding function. See Lemma 1.2 and Theorem 1.3 in \cite{FY20} for the details.

Another interesting property to be explored is the logarithmic convexity for time-fractional Schrödinger equations. For instance, we expect the following estimate to hold for all $0<\alpha \leq \frac{1}{2}$:
\begin{equation}\label{lce}
\|u(t)\| \leq \|u(0)\|^{1-\frac{t}{T}}\|u(T)\|^{\frac{t}{T}}, \quad 0 \leq t \leq T,
\end{equation}
where $u$ is the weak solution to \eqref{fseq}.

For this purpose, we can follow \cite[Theorem 1.1]{Ch24} using the key formula \eqref{kf}, but we need to prove the following conjecture:
\begin{conj}
Let $0<\alpha \leq \frac{1}{2}$. The function
\begin{equation}
f(t)=|E_{2 \alpha,1}\left(-t\right)|^2 + t |E_{2 \alpha, 1+\alpha}\left(-t\right)|^2
\end{equation}
is completely monotone (or at least log-convex) for $t>0$.
\end{conj}
Using Lemma \ref{lmcm}, we know that $|E_{2 \alpha,1}\left(-t\right)|^2$ and $|E_{2 \alpha, 1+\alpha}\left(-t\right)|^2$ are completely monotone for $t>0$ (because $0<2\alpha \le 1$). However, $t |E_{2 \alpha, 1+\alpha}\left(-t\right)|^2$ is not log-convex for $t>0$. Therefore, one needs to consider the whole function $f(t)$. See the plot in Figure \ref{Fig_CM} for $\alpha=0.35$.

\begin{figure}[H]
\centering
\includegraphics[scale=0.6]{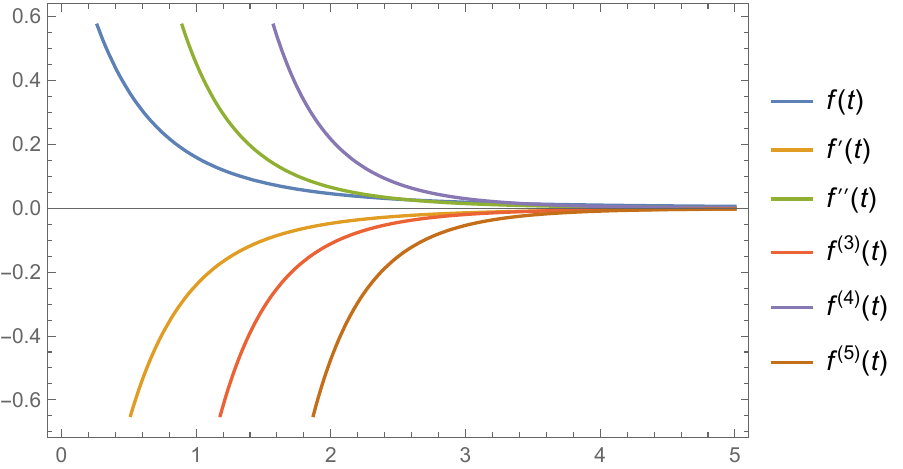}
\caption{First derivatives of $f(t)$ for $\alpha=0.35$.}
\label{Fig_CM}
\end{figure}

\subsection*{Acknowledgment}
The work was supported by Grant-in-Aid for Challenging Research (Pioneering) 21K18142 of Japan Society for the Promotion of Science.

\end{document}